\documentclass[reqno]{amsart}

\usepackage{amsmath}
\usepackage{amsfonts}
\usepackage{amssymb,enumerate}
\usepackage{amsthm}
\usepackage[all]{xy}
\usepackage{hyperref}

\theoremstyle{plain}
\newtheorem{lem}{Lemma}[section]
\newtheorem{cor}[lem]{Corollary}
\newtheorem{prop}[lem]{Proposition}
\newtheorem{thm}[lem]{Theorem}

\newtheorem{intthm}{Theorem}

\theoremstyle{definition}
\newtheorem{defn}[lem]{Definition}
\newtheorem{ex}[lem]{Example}

\newtheorem{questions}[lem]{Questions}

\newtheorem{disc}[lem]{Remark}

\newtheorem*{assumption*}{Assumption}
\newtheorem*{notation*}{Notation}

\newcommand{\depth}{\operatorname{depth}}

\newcommand{\ann}{\operatorname{Ann}}

\newcommand{\Ker}{\operatorname{Ker}}

\newcommand{\ideal}[1]{\mathfrak{#1}}
\newcommand{\m}{\ideal{m}}

\newcommand{\wti}{\widetilde}

\newcommand{\bbn}{\mathbb{N}}

\newcommand{\xra}{\xrightarrow}

\newcommand{\onto}{\twoheadrightarrow}

\renewcommand{\geq}{\geqslant}
\renewcommand{\leq}{\leqslant}

\newcommand{\rad}[1]{\operatorname{rad}(#1)}

\newcommand{\ssm}{\smallsetminus}

\numberwithin{equation}{lem}

\begin{document}

\bibliographystyle{amsplain}

\author{Bethany Kubik}

\address{Bethany Kubik, USMA Department of Mathematical Sciences,
601 Thayer Rd.,
West Point, NY 10996
USA}

\email{bethany.kubik@usma.edu}


\author{Sean Sather-Wagstaff}

\address{Sean Sather-Wagstaff, Department of Mathematics,
NDSU Dept \#2750,
PO Box 6050,
Fargo, ND 58108-6050
USA}

\email{sean.sather-wagstaff@ndsu.edu}

\urladdr{http://www.ndsu.edu/pubweb/\~{}ssatherw/}

\thanks{\today}

\title 
{Path Ideals of Weighted Graphs}



\begin{abstract}
We introduce and study the weighted $r$-path ideal of a weighted graph $G_\omega$,
which is a common generalization of Conca and De Negri's $r$-path ideal for unweighted graphs
and Paulsen and Sather-Wagstaff's edge ideal of the weighted graph.
Over a field, we explicitly describe primary decompositions of these ideals,
and we characterize Cohen-Macaulayness of these ideals for trees (with arbitrary $r$) and complete graphs (for $r=2$).
\end{abstract}

\maketitle


\section*{Introduction}

\begin{assumption*}
Throughout this paper, let $G$ be a (finite, simple) graph with vertex set $V=V(G)=\{v_1,\dots,v_n\}$ of cardinality $n\geq 1$ and edge set $E(G)=E$.  
Let $A$ be a non-zero commutative ring, and set $S=A[X_1,\dots,X_n]$ unless otherwise specified.
Fix an integer $r\in\bbn=\{1,2,\ldots\}$. 
\end{assumption*}

Commutative algebra and combinatorics have a rich history of fruitful interactions.
In this paper, we focus on the connections between commutative algebra and graph theory.
For our purposes, this begins with Villarreal's notion~\cite{villarreal:cmg,villarreal:ma} of the edge ideal associated to the graph $G$, 
which is the ideal $I(G)$ in $S$ ``generated by the edges of $G$''.
Much research has been done on the relations between the combinatorial properties of $G$ and the algebraic properties of $I(G)$;
see, e.g., \cite{francisco:wscmg,francisco:chpg,francisco:eiph,francisco:scmei, martinezbernal:appei,martinezbernal:csreic,morey:eiacp,vantuyl:bgeci,vantuyl:eium2,vantuyl:cmcg}.
For instance, it is straightforward to show that, when $A$ is a field, an irredundant primary decomposition of the ideal $I(G)$
is determined by ``vertex covers'' of the graph $G$. Thus, given decomposition information about $I(G)$, one can deduce
combinatorial information about $G$, and vice versa.

Recently, this construction has been generalized in two different directions relevant to our work.
First, Conca and De Negri~\cite{conca:msgililt} introduce the $r$-path ideal of $G$, when $G$ is a tree.
This is the ideal $I_r(G)$ of $S$ ``generated by the paths in $G$ of length $r$''.
This recovers Villarreal's edge ideal as the special case 
$I_1(G)=I(G)$. See also~\cite{morey:dcmppi} for useful properties of this construction, including a characterization
of the Cohen-Macaulay property of $I_r(G)$.

Next, Paulsen and Sather-Wagstaff~\cite{paulsen:eiwg} introduce the edge ideal of a weighted graph $G_\omega$, i.e.,
a graph $G$ equipped with a function $\omega\colon E\rightarrow\bbn$ that assigns to each edge $e$ of $G$ a weight $\omega(e)$.
The edge ideal $I(G_\omega)$ in this case is generated by the weighted edges of $G_\omega$. In particular, if $1\colon E\to \bbn$
is the constant function $1(e)=1$, then  $I(G_1)=I(G)$.
See Section~\ref{sec130827a} for foundational material about weighted graphs.

In the current paper, we introduce and study a common generalization of these two constructions, the 
\emph{weighted $r$-path ideal} associated to $G_\omega$.
This is the ideal $I_{r}(G_\omega)$ of $S$ that is ``generated by 
the weighted paths of length $r$ of $G$".
$$
I_{r}(G_\omega)=\left(X_{i_1}^{e_{i_1}}\cdots X_{i_{r+1}}^{e_{i_{r+1}}}\left|
\begin{tabular}{l} $v_{i_1}\dots v_{i_{r+1}}$ is a path  in $G$ with $e_{i_1}=\omega(v_{i_1}v_{i_2})$,\! \\
$e_{i_j}=\max\{\omega(v_{i_{j-1}}v_{i_j}),\omega(v_{i_j}v_{i_{j+1}})\}$ for $1<j\leq r$ \!\!\! \\
and $e_{i_{r+1}}=\omega(v_{i_r}v_{i_r+1})$ \\
\end{tabular}
\right.\right)S
$$
As before, this recovers the previous constructions as special cases with $I_r(G_1)=I_r(G)$ and $I_1(G_\omega)=I(G_\omega)$.

We investigate foundational properties of $I_r(G_\omega)$ in
Section~\ref{sec130827b}.
In particular, the following decomposition result
is proved in Theorem~\ref{prop130725}.

\begin{intthm}\label{intprop130725}
Given a weighted graph $G_{\omega}$
one has 
$$I_{r}(G_\omega)=\bigcap_{(W,\sigma)}P_{(W,\sigma)}=\bigcap_{\text{$(W,\sigma)$ min}}\!\!\!\! P_{(W,\sigma)}$$
where the first intersection is taken over all weighted $r$-path  vertex covers of $G_\omega$, and 
the second intersection is taken over all minimal weighted $r$-path  vertex covers of $G_\omega$.
Moreover, the second intersection is irredundant.
\end{intthm}

(See Section~\ref{sec130827a} for definitions of terms like ``weighted $r$-path  vertex cover''.)
When $A$ is a field, this result yields a primary decomposition of $I_r(G_\omega)$.

In Section~\ref{sec130827d} we turn our attention to Cohen-Macaulayness of $I_r(G_\omega)$
when the underlying graph $G$ is a tree.
The main result of this section is the following, which
is proved in Theorem~\ref{thm130831a}.

\begin{intthm}\label{intthm130831a}
Assume that $G_{\omega}$ is a weighted tree and that $A$ is a field. Then the following conditions are equivalent:
\begin{enumerate}[\rm (i)]
\item\label{ithm130831a1} $I_{r}(G_\omega)$ is Cohen-Macaulay;
\item\label{ithm130831a2} $I_{r}(G_\omega)$ is m-unmixed; and
\item\label{ithm130831a3} 
there is a weighted tree $\Gamma_\mu$ and an $r$-path suspension $H_\lambda$ of $\Gamma_\mu$
such that $H_\lambda$ is obtained by pruning a sequence of $r$-pathless leaves from $G_\omega$ and
for all $v_iv_j\in E(\Gamma_\mu)$ we have $\omega(v_iv_j)\leq\min\{\omega(v_iy_{i,1}),\omega(v_jy_{j,1})\}$.
\end{enumerate}
\end{intthm}

Note that this shows that Cohen-Macaulayness of path ideals of weighted trees is independent of the characteristic of $A$.

Section~\ref{sec130827c} is devoted to Cohen-Macaulayness of $I_r(K^n_\omega)$,
where $G=K^n$ is complete, i.e., an $n$-clique.
Note that it is straightforward to show that the edge ideal $I_1(K^n_\omega)$ is always Cohen-Macaulay,
since it is unmixed of dimension 1. On the other hand, the case of $I_r(G_\omega)$ with $r\geq 2$
is  more complicated. We deal with the case $r=2$, the proof of which takes up most of Section~\ref{sec130827c};
see Theorems~\ref{thm130728} and~\ref{thm140730}.

\begin{intthm}\label{inthm130728}
Assume that $n\geq 3$, and let $K^n_\omega$ be a weighted $n$-clique.
Assume that $A$ is a field.
Then the ideal $I_{2}(K^n_\omega)$ is  Cohen-Macaulay if and only if
every induced weighted sub-3-clique $K^3_{\omega'}$ of $K^n_\omega$ has $I_{2}(K^{3}_{\omega'})$ Cohen-Macaulay.
\end{intthm}

As in Theorem~\ref{intthm130831a}, this shows that the Cohen-Macaulay property is characteristic-independent for cliques.
Unlike Theorem~\ref{intthm130831a}, though, it does not say that Cohen-Macaulayness is equivalent to unmixedness.
See Example~\ref{ex140802a} for a weighted 4-clique that is unmixed but not Cohen-Macaulay.

Finally, we note that in Sections~\ref{sec130827a} and~\ref{sec130827b} we deal with a more general situation than the one described in this
introduction. It  uses the following.

\begin{notation*}
Throughout this paper, $G_{\omega}$ is a weighted graph.
Let $\mathcal{P}_2(\bbn)$ denote the set of subsets $U\subset\bbn$ such that $|U|\leq 2$.
Fix a function $f\colon\mathcal{P}_2(\bbn)\to\bbn$, and write
$f\{a,b\}$ in place of $f(\{a,b\})$. For instance, $f$ may be $\max, \min, \gcd,$ or $ \text{lcm}$.
\end{notation*}

\section{Weighted Graphs and Weighted $r$-Path  Vertex Covers}
\label{sec130827a}

In this section, we develop the graph theory used in the rest of the paper, beginning with the unweighted situation.

\begin{defn}
An \emph{$r$-path} in $G$ is a sequence $v_{i_1}\ldots v_{i_{r+1}}$ of distinct vertices in $G$ such that the pair
$v_{i_j}v_{i_{j+1}}$ is an edge in $G$ for $j=1,\ldots,r$. 
\end{defn}

\begin{ex}\label{ex130827a}
Let $G$ be the following tree
$$
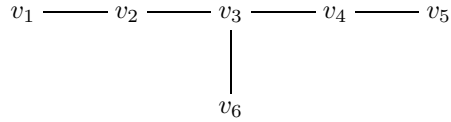
\xymatrix{
v_1\ar@{-}[r]&v_2\ar@{-}[r]&v_3\ar@{-}[r]\ar@{-}[d]&v_4\ar@{-}[r]&v_5\\
&&v_6&&
}$$
and consider the case $r=3$. Then $G$ has four distinct 3-paths, namely $v_1v_2v_3v_4$, $v_1v_2v_3v_6$, $v_2v_3v_4v_5$, and $v_6v_3v_4v_5$.
\end{ex}

The next notion is key for Theorem~\ref{intprop130725} and the rest of the paper.

\begin{defn}\label{defn130617d}
An \emph{r-path vertex cover} of $G$ is a subset $W\subseteq V$ such that for any path $v_{i_1}\dots v_{i_{r+1}}$ of length $r$ in $G$ we have $v_{i_j}\in W$ for some $j$.
In this case, we write that 
$v_{i_j}$ ``covers'' the path.

An $r$-path vertex cover of $G$ is \emph{minimal} if it is minimal with respect to containment, that is,
it does not properly contain another $r$-path vertex cover of $G$.
\end{defn}

For instance, consider the tree $G$ from Example~\ref{ex130827a} with $r=3$.
Then the singleton $\{v_3\}$ is a 3-path vertex cover, since each  3-path in $G$ is covered by $v_3$.
We represent this diagrammatically, as follows.
$$
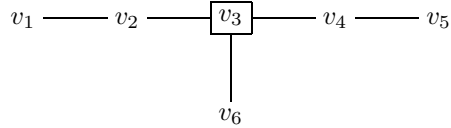
\xymatrix{
v_1\ar@{-}[r]&v_2\ar@{-}[r]&*+[F]{v_3}\ar@{-}[r]\ar@{-}[d]&v_4\ar@{-}[r]&v_5\\
&&v_6&&
}$$
Moreover, this is a minimal 3-path vertex cover of $G$ since $\emptyset$ is not a 3-path vertex cover.
On the other hand, no other singleton is a 3-path vertex cover.
(For instance, the vertex $v_1$ does not cover the path $v_6v_3v_4v_5$.)
However, the set $\{v_1,v_5\}$ is another minimal 3-path vertex cover of $G$.

For graphs represented diagrammatically, we use the diagram for a visual representation of the weight function $\omega$
by decorating each edge $v_iv_j$ with the weight $\omega(v_iv_j)$, as follows.

\begin{ex}\label{ex130827c}
A particular weight function $\omega$ on the tree $G$ from Example~\ref{ex130827a} is represented in the following diagram.
$$
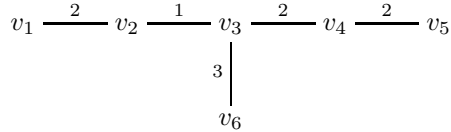
\xymatrix{
v_1\ar@{-}[r]^2&v_2\ar@{-}[r]^{1}&v_3\ar@{-}[r]^2\ar@{-}[d]_3&v_4\ar@{-}[r]^2&v_5\\
&&v_6&&
}$$
For instance, this means that $\omega(v_3v_6)=3$.
\end{ex}

As one may expect, the following definition provides a combinatorial description of   decompositions of ideals constructed from $G_{\omega}$.
See Section~\ref{sec130827b}.

\begin{defn}\label{defn130619b}
Set $\Lambda=\{(W,\sigma)\mid W\subseteq V \text{ and } \sigma:W\rightarrow\bbn\}$.
For each $(W,\sigma)\in\Lambda$, we set $|(W,\sigma)|=|W|$.

An \emph{$f$-weighted $r$-path  vertex cover} of a weighted graph $G_\omega$ is an ordered pair $(W,\sigma)\in\Lambda$ such that for every path 
$v_{i_1}\dots v_{i_{r+1}}$ of length $r$ in $G$, there exists an index $j$ such that $v_{i_j}\in W$ and  one of the  following holds:
\begin{enumerate}[(1)]
\item if $j=1$, then $\sigma(v_{i_j})\leq\omega(v_{i_1}v_{i_2})$;
\item if $j=r+1$, then $\sigma(v_{i_j})\leq\omega(v_{i_r}v_{i_{r+1}})$; or
\item if $1<j\leq r$, then $\sigma(v_{i_j})\leq f\{\omega(v_{i_{j-1}}v_{i_j}), \omega(v_{i_j}v_{i_{j+1}})\}.$
\end{enumerate}
(In particular, when $(W,\sigma)$ is an $f$-weighted $r$-path  vertex cover of $G_\omega$, the set $W$ is an $r$-path  vertex cover of the unweighted graph $G$.)
The number $\sigma(v_{i_j})$ is the \emph{weight} of $v_{i_j}$.  When $v_{i_j}$ satisfies one of the above conditions, we write that it \emph{covers} the path 
$v_{i_1}\dots v_{i_{r+1}}$.
When $f=\max$, we write that $(W,\sigma)$ is a  \emph{weighted $r$-path  vertex cover} of $G_\omega$.
\end{defn}

We represent $f$-weighted $r$-path vertex covers algebraically and diagrammatically, as follows.

\begin{ex}\label{ex130827d}
Consider the weighted tree $G_{\omega}$ from Example~\ref{ex130827c} with $r=3$
and with $f=\max$.
The set $\{v_3\}$ is a 3-path vertex cover of $G$,
and the function $\sigma\colon\{v_3\}\to \bbn$ given by $\sigma(v_3)=2$ yields a weighted 3-path  vertex cover of $G_{\omega}$.
We represent this algebraically and diagrammatically,
by decorating the vertex $v_3$ with the weight $\sigma(v_3)=2$, as follows.
$$
(W,\sigma)=\{v_3^2\}\qquad\qquad\text{and}\qquad\qquad
\xymatrix{
v_1\ar@{-}[r]^2&v_2\ar@{-}[r]^{1}&*+[F]{v_3^2}\ar@{-}[r]^2\ar@{-}[d]_3&v_4\ar@{-}[r]^2&v_5\\
&&v_6&&
}$$
By definition, a function $\sigma'\colon\{v_3\}\to \bbn$  yields a weighted 3-path  vertex cover of $G_{\omega}$ if and only if $\sigma'(v_3)\leq 2$.
Similarly, a decorated set  $\{v_1^{d_1},v_5^{d_5}\}$   describes a weighted 3-path  vertex cover of $G_{\omega}$ if and only if $d_1,d_5\leq 2$.
\end{ex}

\begin{defn}\label{defn130619c}
Given  $(W,\sigma),(W',\sigma')\in\Lambda$, we write $(W',\sigma')\leq(W,\sigma)$ if $W' \subseteq W$ and for all 
$v_i\in W'$ we have $\sigma(v_i)\leq\sigma'(v_i)$.
Naturally, we write $(W',\sigma')<(W,\sigma)$ whenever we have $(W',\sigma')\leq(W,\sigma)$ and $(W',\sigma')\neq(W,\sigma)$.
An $f$-weighted $r$-path  vertex cover $(W,\sigma)$ is \emph{minimal}
if it is minimal  with respect to this ordering, that is,
if there does not exist another $f$-weighted $r$-path  vertex cover $(W',\sigma')$ such that $(W',\sigma')<(W,\sigma)$. 
\end{defn}

\begin{ex}\label{ex130827e}
Consider the weighted tree $G_{\omega}$ from Example~\ref{ex130827c} with $r=3$
and with $f=\max$.
The decorated sets $\{v_3^2\}$ and $\{v_1^{2},v_5^{2}\}$ are minimal weighted 3-path  vertex covers of $G_{\omega}$.
\end{ex}

\begin{ex}\label{rmk130827a}
Given an $r$-path vertex cover $W$ of $G$, it is straightforward to show that the constant function $\sigma\colon W\to\bbn$ with $\sigma(v)=1$
provides an $f$-weighted $r$-path vertex cover $(W,\sigma)$.
\end{ex}

The next two results are for use in the proof of Theorem~\ref{intprop130725}.

\begin{lem}\label{lem130721}
Assume that for all $j\in\bbn$ we have an $f$-weighted $r$-path  vertex cover 
$(W_j,\sigma_j)=\{v_{i_1}^{a_1},\dots,v_{i_p}^{a_p},v_{i_{p+1}}^{b_j}\}$ of $G_\omega$.  If the sequence $\{b_1,b_2,\ldots\}$ is unbounded, then 
$(W,\sigma)=\{v_{i_1}^{a_1},\dots,v_{i_p}^{a_p}\}$ is also an $f$-weighted $r$-path  vertex cover of $G_\omega$.
\end{lem}

\begin{proof}
By assumption, there exists an index $j$ such that $b_j$
is greater than each
of the following numbers: 
$\omega(v_pv_{q})$ for each edge $v_pv_{q}$ in $G$,
and $f(\omega(v_{i}v_j),\omega(v_jv_k))$ for each 2-path $v_iv_jv_k$ in $G$.
It follows that the weighted vertex $v_{i_{p+1}}^{b_j}$ does not cover any $f$-weighted path in $G_\omega$.
Since $(W_j,\sigma_j)$ is an $f$-weighted $r$-path  vertex cover of $G_\omega$, 
it follows that $(W,\sigma)$ is an $f$-weighted $r$-path  vertex cover of $G_\omega$.
\end{proof}

\begin{lem}\label{prop130721}
For every $f$-weighted $r$-path  vertex cover $(W,\sigma)$ of $G_\omega$ there is a minimal $f$-weighted
$r$-path  vertex cover $(W'',\sigma'')$ of $G_\omega$ with  $(W'',\sigma'')\leq(W,\sigma)$.
\end{lem}

\begin{proof}
If $(W,\sigma)$ is a minimal $f$-weighted $r$-path   vertex cover then we are done.  If $(W,\sigma)$ is not  minimal, then either there is a $v_i\in W$ that can be removed or for some $v_i\in W$ the function $\sigma(v_i)$ can be increased.  In the first case, remove vertices from $W$ until the removal of one more vertex creates a path without a vertex to cover it.  This process must terminate in finitely many steps because $W$ is finite.  Let us denote our new $f$-weighted $r$-path  vertex cover as $(W', \sigma')$.  If no vertices are removed, then $(W,\sigma)=(W',\sigma')$.

Lemma~\ref{lem130721} shows that each vertex  $v_i\in W'$ has a bound beyond which one cannot increase the weight on $v_i$ without losing the
$f$-weighted $r$-path   vertex covering property, assuming the weights on the other vertices are held constant. 
In sequence, increase the weight of each vertex to such a bound.
Denote the new ordered pair $(W'',\sigma'')$.  Then, by construction, $(W'',\sigma'')$ is a minimal $f$-weighted $r$-path vertex cover such that $(W'',\sigma'')\leq(W,\sigma)$, and we are done.
\end{proof}

The next result uses $f=\max$.

\begin{lem}\label{lem130902az}
Every minimal weighted $r$-path vertex cover  of $G_\omega$ has cardinality at most $n-1$.
\end{lem}

\begin{proof}
In the case $n\leq r$, the graph $G$ has no $r$-paths, so the empty set describes the unique minimal 
weighted $r$-path vertex cover of $G_\omega$. This has cardinality $0<n$, as desired. 
Thus, for the remainder of the proof, we assume that $n>r$.

Let $(W,\sigma)$ be a weighted $r$-path vertex cover of $G_\omega$. 
We show that, if $|W|=n$, then $(W,\sigma)$ is not minimal.

Assume that $|W|=n$, and write $(W,\sigma)=\{v_1^{e_1},\dots ,v_n^{e_n}\}$.  Reorder the $v_i$ if necessary to assume that $e_1\leq e_2\leq\cdots\leq e_n$.  
We 
show that $v_n^{e_n}$ is superfluous in the vertex cover.  

Suppose by way of contradiction that $v_n^{e_n}$ cannot be removed from $(W,\sigma)$.
This implies that one of the $r$-paths $p$ in $G$ can only be covered by the weighted vertex $v_n^{e_n}$.  
In particular, $p$ must pass through $v_n$,
so assume that $p$ uses the vertices $v_{i_1},\ldots,v_{i_r},v_n$ with $i_1,\ldots,i_r< n$.

As a special case, assume that $p$ has the following form.
$$
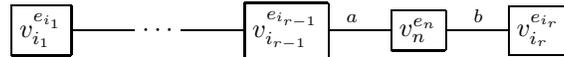
\xymatrix{
*+[F]{v_{i_1}^{e_{i_1}}}\ar@{-}[r]&\cdots\ar@{-}[r]&*+[F]{v_{i_{r-1}}^{e_{i_{r-1}}}}\ar@{-}[r]^a&*+[F]{v_n^{e_n}}\ar@{-}[r]^b&*+[F]{v_{i_r}^{e_{i_r}}}
}$$
By assumption, 
the weighted vertices $v_{i_{r-1}}^{e_{i_{r-1}}}$ and $v_{i_{r}}^{e_{{i_r}}}$ do not cover this path, so we have $e_{i_{r-1}}>a$ and $e_{i_r}>b$.
Also, 
the weighted vertex $v_n^{e_n}$ does cover this path, so we have $e_n\leq a<e_{i_{r-1}}\leq e_n$ or $e_n\leq b<e_{i_r}\leq e_n$, a contradiction.

The general case where $v_n$ is not an endpoint of $p$ is handled similarly. The remaining case where $v_n$
is an endpoint of $p$ is similar, but easier.
\end{proof}

\begin{defn}\label{defn130619a}
A weighted graph $G_\omega$ is \emph{$r$-path unmixed with respect to $f$} if all minimal $f$-weighted $r$-path vertex covers have the same cardinality;
$G_\omega$ is \emph{$r$-path mixed with respect to $f$} is if it is not $r$-path unmixed.
We write that the unweighted graph $G$ is ``$r$-path (un)mixed'' when the trivially weighted graph (with $\omega(e)=1$ for all $e\in E$) is so.
\end{defn}

\section{Weighted Path Ideals and their Decompositions}
\label{sec130827b}

In this section, we introduce and study weighted path ideals.
In particular, we prove Theorem~\ref{intprop130725} from the introduction here.

\begin{defn}\label{defn130619f}
The \emph{$f$-weighted $r$-path ideal} associated to $G_\omega$ is the ideal $I_{r,f}(G_\omega)$ of $S$ that is ``generated by 
the weighted $r$-paths  in $G_{\omega}$".
$$
I_{r,f}(G_\omega)=\left(X_{i_1}^{e_{i_1}}\cdots X_{i_{r+1}}^{e_{i_{r+1}}}\left|
\begin{tabular}{l} $v_{i_1}\dots v_{i_{r+1}}$ is a path  in $G$ with $e_{i_1}=\omega(v_{i_1}v_{i_2})$,\! \\
$e_{i_j}=f\{\omega(v_{i_{j-1}}v_{i_j}),\omega(v_{i_j}v_{i_{j+1}})\}$ for $1<j\leq r$,  \\
and $e_{i_{r+1}}=\omega(v_{i_r}v_{i_r+1})$ \\
\end{tabular}
\right.\right)S
$$
\end{defn}

See Remark~\ref{rmk130827c} for some justification for this definition.

\begin{ex}\label{ex130827f}
Consider the weighted tree $G_{\omega}$ from Example~\ref{ex130827c} with $r=3$
and with $f=\max$.
The 3-path $v_1v_2v_3v_6$ provides one generator  of $I_{3,\max}(G_\omega)$, namely 
$$X_1^{\omega(v_1v_2)}X_2^{\max\{\omega(v_1v_2),\omega(v_2v_3)\}}X_3^{\max\{\omega(v_2v_3),\omega(v_3v_6)\}}X_6^{\omega(v_3v_6)}
=X_1^2X_2^2X_3^3X_6^3.$$
From the remaining 3-paths, we find that 
$$I_{3,\max}(G_\omega)
=(X_1^2X_2^2X_3^3X_6^3, X_1^2X_2^2X_3^2X_4^2, X_2X_3^2X_4^2X_5^2,X_3^3X_4^2X_5^2X_6^3)S.$$
\end{ex}

\begin{disc}\label{rmk130827b}
In the case $r=1$, the ideal $I_{1,f}(G_\omega)$ is the ``weighted edge ideal'' of~\cite{paulsen:eiwg}.
(Note that this is independent of the choice of $f$.)
When $\omega(e)=1$ for all $e\in E$ and $f=\max$, we recover the ``path ideal'' $I_r(G)$ of~\cite{morey:dcmppi,conca:msgililt}.
Also, the special case $f=\max$ yields the ideal $I_r(G_\omega)$ from the introduction.
\end{disc}

\begin{disc}\label{rmk130827c}
Our definition of $I_{r,f}(G_{\omega})$ probably deserves some justification.
Our purpose is to have this definition satisfy the conclusions of Remark~\ref{rmk130827b}.
In order to recover the path ideal of~\cite{morey:dcmppi,conca:msgililt}, the generators should correspond to the
$r$-paths in $G$. To recover the weighted edge ideal of~\cite{paulsen:eiwg} in the case $r=1$,
the generator corresponding to a path $\zeta=v_{i_1}\ldots v_{i_{r+1}}$ should be of the form
$X_{i_1}^{e_{i_1}}\cdots X_{i_{r+1}}^{e_{i_{r+1}}}$ where the exponent $e_{i_j}$ depends on the weights of the 
edges in $\zeta$ that are adjacent to the vertex $v_{i_j}$. For the endpoints $v_{i_1}$ and $v_{i_{r+1}}$,
it seems reasonable to simply use the weight of the only relevant edges, namely, $\omega(v_{i_1}v_{i_2})$ and $\omega(v_{i_{r}}v_{i_{r+1}})$.
However, when $1<j\leq r$, the value of $e_{i_j}$ should depend on both weights $\omega(v_{i_{j-1}}v_{i_{j}})$ and $\omega(v_{i_{j}}v_{i_{j+1}})$.
We entertained several ideas about the ``best'' way to combine these two weights to define $e_{i_j}$, including
max, min, gcd, and lcm. 

Theorem~\ref{prop130725} shows that, from the point of view of decomposing $I_{r,f}(G_{\omega})$
(e.g., computing a primary decomposition of $I_{r,f}(G_{\omega})$, determining unmixedness, etc. when $A$ is a field) 
there is no ``best'' choice for $f$. In other words, every choice for $f$ yields an ideal that we can explicitly decompose. 
(In principle, this explains our choice of condition (3) in Definition~\ref{defn130619b}. While this condition may seem a little strange,
it is the exact condition that  works for our decomposition result.)
On the other hand, our results on Cohen-Macaulayness in Sections~\ref{sec130827d} and~\ref{sec130827c}
indicate that the choice $f=\max$ is somewhat nicer than others we considered,
in that it seems more difficult to characterize Cohen-Macaulayness of $I_{r,f}(G_{\omega})$ when $f\neq\max$.
\end{disc}

In the next definition, recall the notation $\Lambda$ from~\ref{defn130619b}.

\begin{defn}\label{defn130725}
For all $(W,\sigma)\in\Lambda$ we write $P_{(W,\sigma)}=(X_i^{\sigma(v_i)}|v_i\in W)S$.
\end{defn}

One advantage for the algebraic notation from Example~\ref{ex130827d} for elements of $\Lambda$,
is that it explicitly provides generators for the ideal $P_{(W,\sigma)}$.
For instance, with
$(W,\sigma)=\{v_1^{2},v_5^{2}\}$, we have
$$P_{(W,\sigma)}=P_{\{v_1^{2},v_5^{2}\}}=(X_1^{2},X_5^{2})S.$$

\begin{disc}\label{rmk130827d}
It is straightforward to show that the ideals in $S$ of the form $P_{(W,\sigma)}$ are precisely the indecomposable elements of the set of
monomial ideals of $S$. In other words, a monomial ideal $I$ of $S$ is of the form $P_{(W,\sigma)}$ if and only if it satisfies the following:
for all monomial ideals $J_1,J_2$ such that $I=J_1\cap J_2$, one has $I=J_i$ for some $j\in\{1,2\}$.
(In the language of~\cite{rogers:mid}, these are the ``m-irreducible'' monomial ideals of $S$.)
When the coefficient ring $A$ is a field, the ideal $P_{(W,\sigma)}$ is primary with
$\rad{P_{(W,\sigma)}}=(X_i\mid v_i\in W)S$. Hence, when we are working over a field, Theorem~\ref{prop130725}\eqref{prop130725ii}
below gives an irredundant primary decomposition of $I_{r,f}(G_\omega)$.
In general, this is the ``m-irreducible decomposition'' of~\cite{rogers:mid}. 

It is straightforward to show
that every monomial ideal $I$ of $S$ admits a unique irredundant m-irreducible decomposition $I=P_{(W_1,\sigma_1)}\bigcap\cdots\bigcap P_{(W_t,\sigma_t)}$;
uniqueness here is up to reordering of the ideals in the decomposition, and ``irredundant'' means that no ideal in this decomposition is contained in
any other ideal in the decomposition.
We write that $I$ is \emph{m-unmixed} provided that all the $W_i$ in this decomposition have the same cardinality.
We write that $I$ is \emph{m-mixed} provided that it is not m-unmixed.
When we are working over a field, these are equivalent to $I$ being unmixed or mixed, respectively.
\end{disc}

The next result contains  Theorem~\ref{intprop130725} from the introduction.

\begin{thm}\label{prop130725}
\begin{enumerate}[\rm (a)]
\item\label{prop130725i} Given $(W,\sigma)\in\Lambda$, one has $I_{r,f}(G_\omega)\subseteq P_{(W,\sigma)}$ if and only if $(W,\sigma)$
is an $f$-weighted $r$-path  vertex cover of $G_\omega$.
\item\label{prop130725ii} One has decompositions
$$I_{r,f}(G_\omega)=\bigcap_{(W,\sigma)}P_{(W,\sigma)}=\bigcap_{\text{$(W,\sigma)$ min}}\!\!\!\! P_{(W,\sigma)}$$
where the first intersection is taken over all $f$-weighted $r$-path  vertex covers of $G_\omega$, and 
the second intersection is taken over all minimal $f$-weighted $r$-path  vertex covers of $G_\omega$.
Moreover, the second intersection is irredundant.
\end{enumerate}
\end{thm}

\begin{proof}
\eqref{prop130725i} First assume that $(W,\sigma)$ is an
$f$-weighted $r$-path  vertex cover of $G_\omega$, and let $v_{i_1}\cdots v_{i_{r+1}}$ be an $r$-path in $G$.  By definition, there exists a $j\in\{1,\dots,r+1\}$ such that $v_{i_j}\in W$ and one of the following holds:
\begin{description}
\item[$j=1$] we have $\sigma(v_{i_1})\leq\omega(v_{i_1}v_{i_2})=e_{i_1}$;
\item[$j=r+1$] we have $\sigma(v_{i_{r+1}})\leq\omega(v_{i_r}v_{i_{r+1}})=e_{i_{r+1}}$; or
\item[$1<j\leq r$] we have $\sigma(v_{i_j})\leq f\{\omega(v_{i_{j-1}}v_{i_j}),\omega(v_{i_j}v_{i_{j+1}})\}=e_{i_j}.$
\end{description}
In each case we have  $v_{i_j}\in W$ and $\sigma(v_{i_j})\leq e_{i_j}$.  Thus,  $X_{i_j}^{\sigma(v_{i_j})}$ divides $X_{i_j}^{e_{i_j}}$, 
and hence the generator $X_{i_1}^{e_{i_1}}\cdots X_{i_{r+1}}^{e_{i_{r+1}}}$ of $I_{r,f}(G_\omega)$ is in $P_{(W,\sigma)}$.
Since this is true for each $r$-path in $G$, we conclude that $I_{r,f}(G_\omega)\subseteq P_{(W,\sigma)}$.

Conversely, assume that $I_{r,f}(G_\omega)\subseteq P_{(W,\sigma)}$ and let $v_{i_1}\cdots v_{i_{r+1}}$ be an $r$-path in $G$.  By assumption we have $X_{i_1}^{e_{i_1}}\cdots X_{i_{r+1}}^{e_{i_{r+1}}}\in I_{r,f}(G_\omega)\subseteq P_{(W,\sigma)}=(X_i^{\sigma(v_i)}|v_i\in W)$.  Hence there exists an $i$ such that $v_i$ is in $W$ 
and the associated generator $X_i^{\sigma(v_i)}$ divides $X_{i_1}^{e_{i_1}}\cdots X_{i_{r+1}}^{e_{i_{r+1}}}$.  Since $\sigma(v_i)\geq 1$, there exists a $j$ such that $i_j=i$ and $\sigma(v_i)\leq e_{i_j}$.    That is, there exists a $j$ such that $v_{i_j}=v_i\in W$ and $\sigma(v_{i_j})\leq e_{i_j}$.  Since this is true for each $r$-path in $G$,  we conclude that $(W,\sigma)$ is an $f$-weighted $r$-path  vertex cover of $G_\omega$.

\eqref{prop130725ii} 
This follows from Lemma~\ref{prop130721} and part~\eqref{prop130725i},
as in~\cite[Theorem 3.5]{paulsen:eiwg}.
\end{proof}

\begin{cor}\label{prop130728bz}
We have
$\depth(S/I_{r,f}(G_\omega))\geq 1$.
\end{cor}

\begin{proof}
Lemma~\ref{lem130902az} and Theorem~\ref{prop130725}
imply that the maximal ideal $(X_1,\ldots,X_n)S$ is not associated to $I_{2,\max}(K^n_\omega)$, 
hence the desired conclusion.
\end{proof}

\begin{disc}\label{rmk130806}
Remark~\ref{rmk130827d} and Theorem~\ref{prop130725}
imply that $I_{r,f}(G_\omega)$ is m-unmixed if and only if  $G_\omega$ is $r$-path unmixed.
In particular, the $r$-path ideal $I_r(G)$ of~\cite{morey:dcmppi,conca:msgililt} is m-unmixed
if and only if the unweighted graph $G$ is $r$-path unmixed.
\end{disc}

\begin{ex}\label{ex130831a}
Consider the weighted tree $G_{\omega}$ from Example~\ref{ex130827c} with $r=3$
and with $f=\max$.
The ideal $I_{3,\max}(G_\omega)$, computed in Example~\ref{ex130827f}, decomposes irredundantly as follows:
\begin{align*}
I_{3,\max}(G_\omega)=
&(X_3^2)S\bigcap
(X_1^2,X_4^2)S\bigcap
(X_1^2,X_5^2)S\bigcap
(X_2^2,X_4^2)S\bigcap
(X_2^2,X_5^2)S\\
&\bigcap
(X_3^3,X_4^2)S\bigcap
(X_4^2,X_6^3)S\bigcap
(X_2,X_3^3)S\bigcap
(X_2,X_6^3)S.
\end{align*}
If one computes this algebraically (as we did), one can identify all of the minimal weighted $r$-path vertex covers of $G_{\omega}$.
(For instance,  the minimal weighted $r$-path vertex covers
$\{v_3^2\}$ and $\{v_1^{2},v_5^{2}\}$ from Example~\ref{ex130827e}
are visible via the ideals $(X_3^2)S$ and $(X_1^2,X_5^2)S$ in the  decomposition.)
On the other hand, if one is combinatorially inclined, one can first identify all of the minimal $f$-weighted $r$-path vertex covers of $G_{\omega}$,
and then obtain the decomposition from Theorem~\ref{prop130725}. 
\end{ex}

The next lemma is for use in the proof of Theorem~\ref{intthm130831a}.

\begin{lem}\label{lem130725}
If $I_{r,f}(G_\omega)$ is m-unmixed, then $I_r(G)$ is also m-unmixed.
\end{lem}

\begin{proof}
Assume that $I_{r,f}(G_\omega)$ is m-unmixed.  Then there exists an integer $k$ such that  every minimal $f$-weighted $r$-path  
vertex cover $(W,\sigma)$ of $G_\omega$ has $|W|=k$.  Let $W'$ be a minimal $r$-path vertex cover of $G$.  We show that $|W'|=k$.

As we observed in Example~\ref{rmk130827a}, the constant function $\sigma':W'\rightarrow\bbn$ given by $\sigma'(v_i)=1$
yields an $f$-weighted $r$-path  vertex cover $(W',\sigma')$ of $G_\omega$.  Lemma~\ref{prop130721} implies that there exists a minimal $f$-weighted $r$-path vertex cover $(W'',\sigma'')$ of $G_\omega$ such that $(W'',\sigma'')\leq(W',\sigma')$.  By assumption, we have $|W''|=k$.  
By the minimality of $W'$, we have $W''= W'$, hence $|W'|=|W''|=k$.  
\end{proof}

We conclude this section with two lemmas used in the proof of Theorem~\ref{inthm130728}.

\begin{lem}\label{lem130728}
Let $G'_{\omega'}$ denote the weighted subgraph of $G$  induced by $V\ssm\{v_n\}$.  
Set $S'=A[X_1,\ldots,X_{n-1}]$.
Then the natural isomorphism $S/(X_n)S\cong S'$ induces an isomorphism
$$S/(I_{r,\max}(G_\omega)+(X_n)S)\cong S'/I_{r,\max}(G'_{\omega'}). $$
\end{lem}

\begin{proof}
Let $\tau\colon S/(X_n)S\to S'/I_{r,\max}(G'_{\omega'})$ denote the composition of the natural maps $S/(X_n)S\xra\cong S'\onto S'/I_{r,\max}(G'_{\omega'})$.
To show that $\tau$ induces a well-defined epimorphism $\pi\colon S/(I_{r,\max}(G_\omega)+(X_n)S)\onto S'/I_{r,\max}(G'_{\omega'})$,
it suffices to show that each generator of $I_{r,\max}(G_\omega)(S/(X_n)S)$ is in  $\Ker(\tau)$.
Note that the generators of $I_{r,\max}(G_\omega)(S/(X_n)S)$ correspond to the $r$-paths  in $G$ that do not pass through $v_n$.
That is, they correspond to the $r$-paths in $G'$. Since $\omega'(e)=\omega(e)$ for each edge in $G'$, it follows that the
generators of $I_{r,\max}(G_\omega)(S/(X_n)S)$ and $I_{r,\max}(G'_{\omega'})$ corresponding to such a path are equal.
This gives the desired result about $\Ker(\tau)$. A similar argument shows that $\Ker(\tau)=I_{r,\max}(G_\omega)(S/(X_n)S)$, so the induced map $\pi$ is an isomorphism.
\end{proof}

\begin{lem}\label{lem140801a}
The ideal $I_{r,f}(G_\omega)$ can be written as 
$$I_{r,f}(G_\omega)=\sum I_{r,f}(G'_{\omega'})S$$
where the sum is taken over all weighted subgraphs $G'_{\omega'}$ of $G_\omega$ induced by $r+1$ vertices.
(If $G'_{\omega'}$ is induced by $v_{i_1},\ldots,v_{i_{r+1}}$ with $i_1<\cdots<i_{r+1}$, then we consider
$I_{r,f}(G'_{\omega'})$ in the polynomial subring $A[X_{i_1},\ldots,X_{i_{r+1}}]\subseteq S$.)
\end{lem}

\begin{proof}
For the containment $\supseteq$, note that each generator $g$ of $I_{r,f}(G'_{\omega'})S$ is determined by an $r$-path in
$G'$, which is an $r$-path in $G$ with the same weights; hence $g$ is also a generator of $I_{r,f}(G_\omega)$.
For the reverse containment, note that each generator $h$ of $I_{r,f}(G_\omega)$ comes from an $r$-path in $G_\omega$,
and this $r$-path lives in a (unique) induced weighted subgraph $G'_{\omega'}$ of $G_\omega$ on $r+1$ vertices;
thus, $h$ is also a generator of $I_{r,f}(G'_{\omega'})S$.
\end{proof}

\section{Cohen-Macaulay Weighted Trees}
\label{sec130827d}

\begin{assumption*}
Throughout this section, $A$ is a field.
\end{assumption*}

The point of this section is to prove Theorem~\ref{intthm130831a} from the introduction characterizing Cohen-Macaulayness of trees in the
context of weighted path ideals for the function $f=\max$.

\begin{defn}\label{defn130726}
Assume that $v_i$ is a vertex of degree 1 in $G$ that is not a part of any $r$-path in $G$.
We write that $v_i$ is an \emph{$r$-pathless leaf} of $G_{\omega}$.
Let $H_{\lambda}$ be the  weighted subgraph of $G_{\omega}$ induced by the vertex subset $V\ssm\{v_i\}$.
We write that $H_{\lambda}$ is
obtained by \emph{pruning an $r$-pathless leaf} from $G_{\omega}$.
A weighted subgraph $\Gamma_{\mu}$ of $G_{\omega}$ is obtained by \emph{pruning a sequence of $r$-pathless leaves} from $G_\omega$ if there exists a sequence of weighted graphs $G_\omega=G_{\omega^{(0)}}^{(0)},G_{\omega^{(1)}}^{(1)},\dots,G_{\omega^{(l)}}^{(l)}=\Gamma_{\mu}$ such that each $G_{\omega^{(i+1)}}^{(i+1)}$ is obtained by pruning an $r$-pathless leaf from $G_{\omega^{(i)}}^{(i)}$.
\end{defn}

\begin{ex}\label{ex130804}
In the weighted tree $G_{\omega}$ from Example~\ref{ex130827c}, the vertex $v_6$ is a 4-pathless leaf.
Pruning this leaf yields the following weighted path $H_{\lambda}$.
$$\xymatrix{
v_1\ar@{-}[r]^2&v_2\ar@{-}[r]^{1}&v_3\ar@{-}[r]^2&v_4\ar@{-}[r]^2&v_5
}$$
\end{ex}

Next, we  state some consequences of the existence of an $r$-pathless leaf in $G_{\omega}$.

\begin{lem}\label{lem130726x}
Let $H_{\lambda}$ be a weighted graph obtained by pruning a single $r$-pathless leaf $v_i$ from $G_{\omega}$.
\begin{enumerate}[\rm(a)]
\item \label{lem130726x1}
The set of $r$-paths in $G$ is the same as the set of $r$-paths in $H$.
\item \label{lem130726x2}
Assume that $(W,\sigma)$ is an $f$-weighted $r$-path vertex cover of $G_{\omega}$ such that $v_i\in W$.
Set $W'=W\ssm\{v_i\}$ and $\sigma'=\sigma|_{W'}$. Then $(W',\sigma')$ is an $f$-weighted $r$-path vertex cover of $G_{\omega}$.
\item \label{lem130726x3}
The minimal $f$-weighted $r$-path vertex covers of $G_{\omega}$ are the same as the 
minimal $f$-weighted $r$-path vertex covers of $H_{\lambda}$, so $G_\omega$ is $r$-path unmixed with respect to $f$ if and only if $H_{\lambda}$ is so.
\end{enumerate}
\end{lem}

\begin{proof}
\eqref{lem130726x1}
This follows by definition of $H$ since no $r$-paths in $G$ pass through $v_i$.

\eqref{lem130726x2}
Since no $r$-paths pass through $v_i$, this vertex does not cover any $r$-paths, so it can be removed.

\eqref{lem130726x3}
Combining parts~\eqref{lem130726x1}
and~\eqref{lem130726x2}, we conclude that the $f$-weighted $r$-path vertex covers of $H_{\lambda}$
are exactly the $f$-weighted $r$-path vertex covers $(W,\sigma)$ of $G_{\omega}$
such that $v_i\notin W$.
The desired conclusion about minimal elements now follows.
\end{proof}

The next definition is key for Theorem~\ref{intthm130831a}.

\begin{defn}\label{defn130721}
Let $\Gamma_{\mu}$ be a weighted graph. 
The \emph{$r$-path suspension} of the unweighted graph $\Gamma$ is the graph obtained by adding a new path of length $r$ to each vertex of $\Gamma$.
The new $r$-paths are called \emph{$r$-whiskers}.
A weighted graph $H_{\lambda}$ is a \emph{weighted $r$-path suspension} of $\Gamma_{\mu}$ provided that the unweighted graph $H$
is an $r$-path suspension of  $\Gamma$.
\end{defn}

\begin{ex}\label{ex130831b}
The weighted tree $G_{\omega}$ from Example~\ref{ex130827c} is a weighted $1$-path suspension
of the following weighted path. 
$$
\xymatrix{v_2\ar@{-}[r]^{1}&v_3\ar@{-}[r]^2&v_4
}$$
Examples of weighted $r$-path suspensions  of $G_{\omega}$ itself
are given by the following, where the edges of $G$ are drawn double for emphasis.
\begin{equation*}
\xymatrix@C=7.5mm@R=8mm{&&y_{1,1}\ar@{-}[d]_4&y_{2,1}\ar@{-}[d]_3&y_{3,1}\ar@{-}[d]_3&y_{4,1}\ar@{-}[d]_4&y_{5,1}\ar@{-}[d]_2\\
r=1&&v_1\ar@{=}[r]^2&v_2\ar@{=}[r]^{1}&v_3\ar@{=}[r]^2\ar@{=}[d]_3&v_4\ar@{=}[r]^2&v_5&&(G'_{\omega'})\\
&&&y_{6,1}\ar@{-}[r]^{2}&v_6&&\\
&&y_{1,2}\ar@{-}[d]_3&y_{2,2}\ar@{-}[d]_3&y_{3,2}\ar@{-}[d]_5&y_{4,2}\ar@{-}[d]_4&y_{5,2}\ar@{-}[d]_2& \\
&&y_{1,1}\ar@{-}[d]_4&y_{2,1}\ar@{-}[d]_3&y_{3,1}\ar@{-}[d]_3&y_{4,1}\ar@{-}[d]_4&y_{5,1}\ar@{-}[d]_2& \\
r=2&&v_1\ar@{=}[r]^2&v_2\ar@{=}[r]^{1}&v_3\ar@{=}[r]^2\ar@{=}[d]_3&v_4\ar@{=}[r]^2&v_5&&(G''_{\omega''})\\
&&y_{6,2}\ar@{-}[r]^3&y_{6,1}\ar@{-}[r]^{200}&v_6&&
}
\end{equation*}
\end{ex}

\begin{disc}\label{disc130831a}
A weighted graph $H_{\lambda}$ 
is an $r$-path suspension of another weighted graph $\Gamma_{\mu}$ if and only if $H$ has a sequence of pair-wise disjoint paths $p_1,p_2,\dots,p_{\beta}$  of length $r$ 
such that (after appropriately renaming the vertices of $H$) the 
the vertices of each $p_i$ can be ordered as $v_i,y_{i,1},\dots,y_{i,r}$ where $\deg(y_{i,k})=2$ for  $k=1,\ldots,r-1$, and $\deg(y_{i,r})=1$,
such that $V(H)=\{v_1,y_{1,1},\dots,y_{1,r},\ldots,v_\beta,y_{\beta,1},\dots,y_{\beta,r}\}$.
In this case, $\Gamma$ is the induced subgraph of $H$  associated to the subset $\{v_1,\ldots,v_{\beta}\}\subseteq V$.
When this is the case, we write  $S=A[X_1,Y_{1,1},\dots,Y_{1,r},\ldots,X_\beta,Y_{\beta,1},\dots,Y_{\beta,r}]$
instead of $A[X_1,\ldots,X_n]$ for the polynomial ring containing $I_{r,\max}(H_{\lambda})$.
\end{disc}

The following proposition contains one implication of Theorem~\ref{intthm130831a}.

\begin{prop}\label{lem130726}
Let $H_{\lambda}$ be an $r$-path suspension of the weighted graph $\Gamma_\mu$, with notation as in Remark~\ref{disc130831a},
such that for all $v_iv_j\in E(\Gamma)$ we have $\omega(v_iv_j)\leq \min\{\omega(v_iy_{i,1}),\omega(v_jy_{j,1})\}$.  Then  $I_{r,\max}(H_{\lambda})$ is 
Cohen-Macaulay.
\end{prop}

\begin{proof}
As in the proof of~\cite[Lemma 5.3]{paulsen:eiwg}, we polarize the ideal $I:=I_{r,\max}(H_{\lambda})$ to obtain a new ideal $\wti I$ in a new polynomial ring $\wti S$.
We then show that $\wti I$ is the polarization of another monomial ideal $J$ in another polynomial ring $T$ such that $T/J$ is 
artinian. In particular, $T/J$ is Cohen-Macualay. Since $T/J$ and $S/I$ are graded specializations of $\wti S/\wti I$, it follows that
$\wti S/\wti I$ and $S/I$ are also Cohen-Macaulay.

In preparation, we set some notation 
\begin{align*}
a_i:=&\omega(v_iy_{i,1}) && \mbox{ for } i=1,\dots,\beta \\
a_{i,1}:=&\max\{\omega(v_iy_{i,1}),\omega(y_{i,1}y_{i,2})\} && \mbox{ for } i=1,\dots,\beta \\
a_{i,j}:=&\max\{\omega(y_{i,j-1}y_{i,j}),\omega(y_{i,j}y_{i,j+1})\} && \mbox{ for } i=1,\dots,\beta \mbox{ and } j=2,\dots,r-1 \\
a_{i,r}:=&\omega(y_{i,r-1}y_{i,r}) && \mbox{ for } i=1,\dots,\beta \\
t_{i,j}=&\omega(y_{i,j-1}y_{i,j})  && \mbox{ for } i=1,\dots,\beta \mbox{ and } j=2,\dots,r-1 \\
b_{p,q,r}=&\max\{\omega(v_pv_q),\omega(v_qv_r)\} && \mbox{ for all $2$-paths } v_pv_qv_r  \mbox{ in $\Gamma$} \\
c_{i,j}=&\omega(v_iv_j) && \mbox{ for all edges } v_iv_j  \mbox{ in $\Gamma$}
\end{align*}
The polynomial ring $\wti S$ has coefficients in $A$ with the following list of variables.
\begin{align*}
&X_{1,1},\ldots, X_{1,a_1},Y_{1,1,1},\ldots, Y_{1,1,a_{1,1}},Y_{1,2,1},\ldots, Y_{1,2,a_{1,2}},\ldots, Y_{1,r,1},\ldots, Y_{1,r,a_{1,r}},\\
&X_{2,1},\ldots, X_{2,a_2},Y_{2,1,1},\ldots, Y_{2,1,a_{2,1}},Y_{2,2,1},\ldots, Y_{2,2,a_{2,2}},\ldots, Y_{2,r,1}\ldots Y_{2,r,a_{2,r}},\ldots, \\
&X_{\beta,1},\ldots, X_{\beta,a_\beta},Y_{\beta,1,1},\ldots, Y_{\beta,1,a_{\beta,1}},Y_{\beta,2,1},\ldots, Y_{\beta,2,a_{\beta,2}},\ldots, Y_{\beta,r,1}\ldots Y_{\beta,r,a_{\beta,r}}
\end{align*}

To polarize the ideal $I$, we need to polarize the generators, which correspond to the $r$-paths in $H$.
There are four types of $r$-paths in $H$: paths completely contained in an $r$-whisker (that is,  exactly  an $r$-whisker); paths partially in a $r$-whisker and partially in $\Gamma$; paths that start in a $r$-whisker, run through part of $\Gamma$, then end in another $r$-whisker; and paths that are completely in $\Gamma$.

First, consider an $r$-whisker $v_iy_{i,1}\dots y_{i,r}$.  
The generator associated to this path in $I$ is $X_i^{a_i}Y_{i,1}^{a_{i,1}}Y_{i,2}^{a_{i,2}}\cdots Y_{i,r}^{a_{i,r}}$.
When we polarize this generator, we obtain the following generator of $\wti I$.
\begin{equation}\label{eq130901a}
X_{i,1}\cdots X_{i,a_i}Y_{i,1,1}\cdots Y_{i,1,a_{i,1}}Y_{i,2,1}\cdots Y_{i,2,a_{i,2}}\cdots Y_{i,r,1}\cdots Y_{i,r,a_{i,r}}
\end{equation}

Next, consider an $r$-path $v_{i_1}v_{i_2}\cdots v_{i_p}v_jy_{j,1}\cdots y_{j,k}$ that starts in $\Gamma$ and ends in an $r$-whisker.  
Note that here we have $p+k=r$.  
The generator of $I$ associated to this path is $X_{i_1}^{c_{i_1,i_2}}X_{i_2}^{b_{i_{1},i_2,i_3}}\cdots X_{i_p}^{b_{i_{p-1},i_p,j}}X_j^{a_j}Y_{j,1}^{a_{j,1}}\cdots Y_{j,k-1}^{a_{j,k-1}}Y_{j,k}^{t_{j,k}}$.
When we polarize this generator for $I$, we obtain the next generator for $\wti I$.
\begin{multline}\label{eq130901b}
X_{i_1,1}\cdots X_{i_1,c_{i_1,i_2}}X_{i_2,1}\cdots X_{i_2,b_{i_{1},i_2,i_3}}\cdots X_{i_p,1}\cdots X_{i_p,b_{i_{p-1},i_p,j}}\\
\cdot X_{j,1}\cdots X_{j,a_j}Y_{j,1,1}\cdots  Y_{j,1,a_{j,1}}\cdots Y_{j,k-1,1}\cdots Y_{j,k-1,a_{j,k-1}}Y_{j,k,1}\cdots Y_{j,k,t_{j,k}}
\end{multline}
Observe that the assumption  $\omega(v_iv_j)\leq \min\{\omega(v_iy_{i,1}),\omega(v_jy_{j,1})\}$ for all $v_iv_j\in E(\Gamma)$ implies that we have 
$c_{i_1,i_2}\leq a_{i_1}$. Similarly, we have
$b_{i_{1},i_2,i_3}\leq a_{i_2}$, and the inequality $t_{j,k}\leq a_{j,k}$ is by construction.
Thus, the generator~\eqref{eq130901b} is in $\wti S$.

Next, consider an $r$-path $y_{j,q}\dots y_{j,1}v_jv_{m_1}\dots v_{m_l}v_iy_{i,1}\dots y_{i,p}$ that starts in an $r$-whisker, runs through part of $\Gamma$, and  ends in another $r$-whisker.  Note that we have $l\geq 0$ and  $q+l+p+1=r$. The generator in $I$ associated to this type of path is
the following.
$$Y^{t_{j,q}}_{j,q}Y_{j,q-1}^{a_{j,q-1}}\cdots Y_{j,1}^{a_{j,1}}X_j^{a_j}X_{m_1}^{b_{j,m_1,m_2}}\cdots X_{m_l}^{b_{m_{l-1},m_l,i}}X_i^{a_i}Y_{i,1}^{a_{i,1}}\cdots Y_{i,p-1}^{a_{i,p-1}}Y_{i,p}^{t_{i,p}}$$  
When we polarize this generator we obtain the next generator for $\wti I$.
\begin{multline}\label{eq130901c}
Y_{j,q,1}\cdots Y_{j,q,t_{j,q}}Y_{j,q-1,1}\cdots Y_{j,q-1,a_{j,q-1}}\cdots Y_{j,1,1}\cdots Y_{j,1,a_{j,1}}\\
\cdot X_{j,1}\cdots X_{j,a_j} X_{m_1,1}\cdots X_{m_1,b_{j,m_1,m_2}}\cdots X_{m_l,1}\cdots X_{m_l,b_{m_{l-1},m_l,i}}X_{i,1}\cdots X_{i,a_i}\\
\cdot Y_{i,1,1}\cdots Y_{i,1,a_{i,1}}\cdots Y_{i,p-1,1}\cdots Y_{i,p-1,a_{i,p-1}}Y_{i,p,1}\cdots Y_{i,p,t_{i,p}}
\end{multline}

For the last type of generator, consider an $r$-path $v_{i_1}\dots v_{i_{r+1}}$ entirely in $\Gamma$.  The generator in $I$ associated to this path is 
the following.
$$X_{i_1}^{c_{i_1,i_2}}X_{i_2}^{b_{i_1,i_2,i_3}}\cdots X_{i_r}^{b_{i_{r-1},i_r,i_{r+1}}}X_{i_{r+1}}^{c_{i_r,i_{r+1}}}$$
When we polarize this generator we obtain the next generator for $\wti I$.
\begin{equation}\label{eq130901d}
X_{i_1,1}\cdots X_{i_1,c_{i_1,i_2}}X_{i_2,1}\cdots X_{i_2,b_{i_1,i_2,i_3}}\cdots X_{i_{r+1},1}\cdots X_{i_{r+1},c_{i_r,i_{r+1}}}
\end{equation}

Set $T=A[X_{1,1},\ldots,X_{\beta,1}]$, and let $J$ be the monomial ideal of $T$ with the following generators.
For each $r$-whisker $v_iy_{i,1}\dots y_{i,r}$, include the following generator.
\begin{equation}\label{eq130901e}
X_{i,1}^{a_i+a_{i,1}+\cdots+a_{i,r}}
\end{equation}
For each $r$-path $v_{i_1}v_{i_2}\cdots v_{i_p}v_jy_{j,1}\cdots y_{j,k}$ that starts in $\Gamma$ and ends in an $r$-whisker, include the next generator. 
\begin{equation}\label{eq130901f}
X_{i_1,1}^{c_{i_1,i_2}}X_{i_2,1}^{b_{i_{1},i_2,i_3}}\cdots X_{i_p,1}^{b_{i_{p-1},i_p,j}}X_{j,1}^{a_j+a_{j,1}+\cdots +a_{j,k-1}+t_{j,k}}
\end{equation}
For each $r$-path $y_{j,q}\dots y_{j,1}v_jv_{m_1}\dots v_{m_l}v_iy_{i,1}\dots y_{i,p}$ that starts in an $r$-whisker, runs through part of $\Gamma$, and  ends in another $r$-whisker, include the next generator.
\begin{equation}\label{eq130901g}
X_{j,1}^{t_{j,q}+a_{j,q-1}+\cdots +a_{j,1}+a_j}
X_{m_1,1}^{b_{j,m_1,m_2}}\cdots X_{m_l,1}^{b_{m_{l-1},m_l,i}}X_{i,1}^{a_i+a_{i,1}+\cdots +a_{i,p-1}+t_{i,p}}
\end{equation}
For each $r$-path $v_{i_1}\dots v_{i_{r+1}}$ entirely in $\Gamma$, include the next generator.
\begin{equation}\label{eq130901h}
X_{i_1,1}^{c_{i_1,i_2}}X_{i_2,1}^{b_{i_1,i_2,i_3}}\cdots X_{i_{r+1},1}^{c_{i_r,i_{r+1}}}
\end{equation}
It is straightforward to show that the polarization of $J$ is exactly $\wti I$: for $n=1,2,3,4$, the polarization of the generator (\ref{lem130726}.$n+4$) of $J$ is exactly the generator 
(\ref{lem130726}.$n$) of $\wti I$.
Since $J$ contains a power of each of the variables in $T$, namely~\eqref{eq130901e}, we conclude that $T/J$ is artinian.
Thus, the first paragraph of this proof implies that $S/I$ is Cohen-Macaulay.
\end{proof}

\begin{ex}\label{ex130831d}
For the weighted graph $G_{\omega}$ in Example~\ref{ex130827c}, Proposition~\ref{lem130726} shows that $I_{1,\max}(G_{\omega})$ is Cohen-Macaulay,
and similarly for $I_{2,\max}(G''_{\omega''})$
in Example~\ref{ex130831b}. 
See also Examples~\ref{ex130831e} and~\ref{ex130901a}.
\end{ex}

Note that the ideals $I_{r,\max}(G_\omega)$ and $I_{r,\max}(H_{\lambda})$ in the next result live in different polynomial rings.

\begin{lem}\label{lem130726a}
Let $H_{\lambda}$ be a weighted graph obtained by pruning a sequence of $r$-pathless leaves from $G_\omega$.
\begin{enumerate}[\rm(a)]
\item \label{lem130726a1}
The ideals $I_{r,\max}(G_\omega)$ and $I_{r,\max}(H_{\lambda})$ have the same generators. 
\item \label{lem130726a2}
The ideal $I_{r,\max}(G_\omega)$ is m-unmixed if and only if $I_{r,\max}(H_{\lambda})$ is so.
\item \label{lem130726a3}
The ideal $I_{r,\max}(G_\omega)$ is Cohen-Macaulay if and only if $I_{r,\max}(H_{\lambda})$ is so.
\end{enumerate}
\end{lem}

\begin{proof}
Arguing by induction on the number of $r$-pathless leaves being pruned from $G_{\omega}$, we assume that 
$H_{\lambda}$ is
obtained by pruning a single $r$-pathless leaf $v_i$ from $G_{\omega}$.

\eqref{lem130726a1}
By Lemma~\ref{lem130726x}\eqref{lem130726x1}, the set of $r$-paths in $G$ is the same as the set of $r$-paths in $H$, and 
$\lambda(e)=\omega(e)$ for each edge $e\in E(H)\subseteq E(G)$.
The claim about the generators now follows directly.  

\eqref{lem130726a2}
This follows from Theorem~\ref{prop130725}\eqref{prop130725ii} and Lemma~\ref{lem130726x}\eqref{lem130726x3}.

\eqref{lem130726a3}
Part~\eqref{lem130726a1} implies that
$(S'/I_{r,\max}(H_{\lambda}))[X]\cong S/I_{r,\max}(G_\omega)$, 
where $S':=A[X_1,\ldots,X_{i-1},X_{i+1},\ldots,X_n]$.
It follows 
that $S/I_{r,\max}(G_\omega)$ is Cohen-Macaulay if and only if $S'/I_{r,\max}(H_{\lambda})$ is Cohen-Macaulay,
as desired.  
\end{proof}

The next result compares directly to Theorem~\ref{intthm130831a} from the introduction, though it does not assume that $G$ is a tree.

\begin{prop}\label{prop130725b}
Assume that $H_\lambda$ is obtained by pruning a sequence of $r$-pathless leaves from $G_\omega$ 
and that $H_\lambda$ is an $r$-path suspension of a weighted graph $\Gamma_\mu$.  
With notation as in Remark~\ref{disc130831a}, the following conditions are equivalent:
\begin{enumerate}[\rm (i)]
\item\label{prop130725bi} $I_{r,\max}(G_\omega)$ is Cohen-Macaulay;
\item\label{prop130725bii} $I_{r,\max}(G_\omega)$ is m-unmixed; and
\item\label{prop130725biii} for all $v_iv_j\in E(\Gamma_\mu)$ we have $\omega(v_iv_j)\leq\min\{\omega(v_iy_{i,1}),\omega(v_jy_{j,1})\}$.
\end{enumerate}
\end{prop}

\begin{proof}
The case $r=1$ is handled in~\cite[Theorem 5.7]{paulsen:eiwg}, so we assume that $r\geq 2$ for the remainder of the proof.
The implication~\eqref{prop130725bi}$\implies$\eqref{prop130725bii} always holds.  

\eqref{prop130725bii}$\implies$\eqref{prop130725biii}  Assume that $I_{r,\max}(G_\omega)$ is m-unmixed.
It follows from Lemma~\ref{lem130726a}\eqref{lem130726a2} that $I_{r,\max}(H_\lambda)$ is also unmixed.  
From an analysis of the $r$-paths of $H$ as in the proof of Proposition~\ref{lem130726},
it is straightforward to show that $V(\Gamma_\mu)$ is a minimal $r$-path vertex cover of $H$.
(It covers all the paths, and the $r$-whiskers show that it is minimal.)
Let $\tau\colon V(\Gamma_\mu)\to\bbn$ be the constant function $\tau(v_i)=1$.
Lemma~\ref{prop130721} implies that there is a minimal weighted $r$-path vertex cover $(W'',\sigma'')$ of $H_{\lambda}$
such that $(W'',\sigma'')\leq (V(\Gamma_\mu),\tau)$. The minimality of $V(\Gamma_\mu)$ implies that
$W''=V(\Gamma_\mu)$, so $(V(\Gamma_\mu),\sigma'')$ is a minimal weighted $r$-path vertex cover  of $H_{\lambda}$.
The unmixedness condition implies that every minimal weighted $r$-path vertex cover of $H_{\lambda}$ has cardinality $|V(\Gamma_\mu)|$.

We proceed by by contradiction.
Suppose that there is an edge $v_iv_j\in E(\Gamma_\mu)$ such that $\omega(v_iv_j)>\min\{\omega(v_iy_{i,1}),\omega(v_jy_{j,1})\}$.  
We produce a contradiction by showing that there exists a minimal weighted $r$-path vertex cover $(W,\sigma)$ of $H_\lambda$ such that $|W|>|V(\Gamma_\mu)|$.  
Assume by symmetry that 
$$a:=\omega(v_iy_{i,1})=\min\{\omega(v_iy_{i,1}),\omega(v_jy_{j,1})\}<\omega(v_iv_j)=:b.$$
Set $c=\omega(v_jy_{j,1})$ and $a':=\omega(y_{i,r-1}y_{i,r})$ and $c':=\omega(y_{j,r-1}y_{j,r})$.
The following digram 
(where the column represents $\Gamma$, and the rows represent the $r$-whiskers in $H$)
is our guide for constructing an approximation of $(W,\sigma)$.
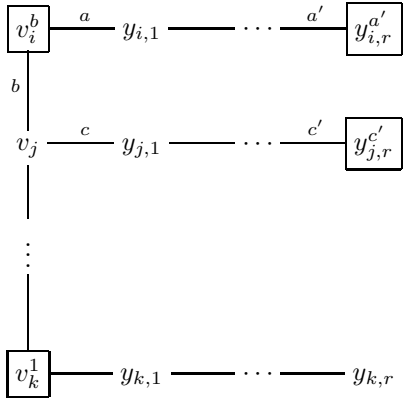
$$\xymatrix{
*+[F]{v_i^b}\ar@{-}[d]_b\ar@{-}[r]^a&y_{i,1}\ar@{-}[r]&\cdots\ar@{-}[r]^{a'}
&*+[F]{y_{i,r}^{a'}}\\
v_j\ar@{-}[d]\ar@{-}[r]^c&y_{j,1}\ar@{-}[r]&\cdots\ar@{-}[r]^{c'}&*+[F]{y_{j,r}^{c'}}\\
\vdots\ar@{-}[d]&&&\\
*+[F]{v_k^{1}}\ar@{-}[r]&y_{k,1}\ar@{-}[r]&\cdots\ar@{-}[r]&y_{k,r}
}$$
Set $W=\{v_k|k\neq j\}\cup\{y_{i,r},y_{j,r}\}$ and define $\sigma:W\rightarrow\bbn$ by 
\begin{align*}
\sigma(v_k)&=\begin{cases} 1\makebox{ if $k\neq i$}\\b \makebox{ if $k=i$}\end{cases}\\
\sigma(y_{i,r})&=a'\\
\sigma(y_{j,r})&=c'.
\end{align*}
It is straightforward to show that $(W,\sigma)$ is a weighted $r$-path vertex cover of $H_\lambda$.
Lemma~\ref{prop130721} provides a minimal weighted
$r$-path  vertex cover $(W',\sigma')$ of $G_\omega$ such that $(W',\sigma')\leq(W,\sigma)$.

We claim that $W'=W$. 
(This then yields the promised contradiction, completing the proof of this implication.)
To this end, first note that we have $W'\subseteq W$, by assumption.
So, we need to show that $W'\supseteq W$.
We 
cannot remove the vertex $y_{j,r}$ from $W$,
since that would leave the $r$-path $v_jy_{j,1}\dots y_{j,r}$ uncovered.
Thus, we have $y_{j,r}\in W'$.
Similarly, for $k\neq i,j$ the vertex $v_k$ cannot be removed,
so $v_k\in W'$.
If we remove the vertex $v_i$, the $r$-path $v_jv_iv_{i,1}\dots v_{i,r-1}$ is not covered,
so $v_i\in W'$.
Since $\sigma(v_i)=b>a$, the vertex $v_i$ does not cover the $r$-path $v_iy_{i,1}\dots y_{i,r}$.
It follows that the vertex $y_{i,r}$ cannot be removed.
Thus, we have
$y_{i,r}\in W'$, and it follows that $W'=W$, as claimed.

\eqref{prop130725biii}$\implies$\eqref{prop130725bi} Assuming condition~\eqref{prop130725biii}, Proposition~\ref{lem130726} 
implies that $I_{r,\max}(H_\lambda)$ is Cohen-Macaulay,
so Lemma~\ref{lem130726a}\eqref{lem130726a3}  implies that $I_{r,\max}(G_\omega)$ is as well.
\end{proof}

The next result contains Theorem~\ref{intthm130831a} from the introduction.

\begin{thm}\label{thm130831a}
Assume that $G_{\omega}$ is a weighted tree. Then the following conditions are equivalent:
\begin{enumerate}[\rm (i)]
\item\label{thm130831a1} $I_{r,\max}(G_\omega)$ is Cohen-Macaulay;
\item\label{thm130831a2} $I_{r,\max}(G_\omega)$ is m-unmixed; and
\item\label{thm130831a3} 
there is a weighted tree $\Gamma_\mu$ and an $r$-path suspension $H_\lambda$ of $\Gamma_\mu$
such that $H_\lambda$ is obtained by pruning a sequence of $r$-pathless leaves from $G_\omega$ and
for all $v_iv_j\in E(\Gamma_\mu)$ we have $\omega(v_iv_j)\leq\min\{\omega(v_iy_{i,1}),\omega(v_jy_{j,1})\}$.
\end{enumerate}
When  $G_{\omega}$ satisfies the above equivalent conditions,
the graph $H$ can be constructed by pruning $r$-pathless leaves from $G$ until no more $r$-pathless leaves remain.
\end{thm}

\begin{proof}
The implications $\eqref{thm130831a3}\implies\eqref{thm130831a1}\implies\eqref{thm130831a2}$ are from
Proposition~\ref{prop130725b}.
For the implication $\eqref{thm130831a2}\implies\eqref{thm130831a3}$, assume that $I_{r,\max}(G_\omega)$ is m-unmixed.
Since $G$ is finite, prune a sequence of $r$-pathless leaves from $G_{\omega}$ to obtain a weighted subgraph $H_{\lambda}$
that has no $r$-pathless leaves.
Lemma~\ref{lem130726a}\eqref{lem130726a2} implies that $I_{r,\max}(H_{\lambda})$ is m-unmixed,
so Lemma~\ref{lem130725} implies that $I_{r}(H)$ is m-unmixed.
Thus, $H$ is an  $r$-path suspension of a  tree $\Gamma$ by~\cite[Theorem 3.8 and Remark 3.9]{morey:dcmppi}.
Finally, Proposition~\ref{prop130725b} implies that 
$\omega(v_iv_j)\leq\min\{\omega(v_iy_{i,1}),\omega(v_jy_{j,1})\}$ for all $v_iv_j\in E(\Gamma_\mu)$.
\end{proof}

\begin{ex}\label{ex130831e}
Consider the weighted graph $G_{\omega}$ 
in Example~\ref{ex130827c}.
Then $I_{r,\max}(G_{\omega})$ is Cohen-Macaulay if and only if $r\neq 2,3$, as follows.
Example~\ref{ex130831d} deals with the case $r=1$.

For $r>5$, the ideal $I_{r,\max}(G_{\omega})$ is trivially Cohen-Macaulay since $G$ has no $r$-paths.
(One can also deduce this from Lemma~\ref{lem130726a} since every leaf is $r$-pathless.)

This graph has a single 4-path, so $S/I_{4,\max}(G_{\omega})$ is a hypersurface, hence Cohen-Macaulay.
One can also deduce this from Theorem~\ref{thm130831a} by pruning the 4-pathless leaf $v_6$
to obtain the weighted 4-path $H_{\lambda}$ in Example~\ref{ex130804}. 
Since $H_{\lambda}$ is a 4-path suspension of the trivial
graph $v_1$, the desired conclusion follows from Theorem~\ref{thm130831a}.

For $r=2,3$, the ideal $I_{r,\max}(G_{\omega})$ is not Cohen-Macaulay by Theorem~\ref{thm130831a}. To see this, observe that
$G$ does not have any $r$-pathless leaves and is not an $r$-suspension for $r=2,3$.
\end{ex}

\begin{ex}\label{ex130901a}
Arguing as in Example~\ref{ex130831e}, we have the following for the weighted graphs $G'_{\omega'}$ and $G''_{\omega''}$
of Example~\ref{ex130831b}.
The ideal $I_{r,\max}(G'_{\omega'})$ is Cohen-Macaulay if and only if $r\geq 6$, and $I_{r,\max}(G''_{\omega''})$ is Cohen-Macaulay
 if and only if $r\neq 1,3,4,5,6,7$.
\end{ex}

\section{Cohen-Macaulay Weighted Complete Graphs when $r=2$}
\label{sec130827c}

\begin{assumption*}
Throughout this section, $K^n_{\omega}$ is a weighted $n$-clique, and $A$ is a field.
\end{assumption*}

In this section, we prove Theorem~\ref{inthm130728} from the introduction characterizing Cohen-Macaulayness of $n$-cliques in the
context of weighted path ideals for the function $f=\max$ with $r=2$.
We begin with two results about arbitrary $f$ and $r$.
Note that the assumption $r< n$ causes no loss of information since, when $r\geq n$, we have $I_{r,f}(G_{\omega})=0$.

\begin{lem}\label{lem130728a}
If $(W,\sigma)$ is an $f$-weighted $r$-path vertex cover for $K^n_\omega$ where $r< n$, then $|W|\geq n-r.$
\end{lem}

\begin{proof}
Suppose that $|W|<n-r$ and assume  that  $v_{i_1},\dots,v_{i_{r+1}}\notin W$.  Then the path $v_{i_1}\dots v_{i_{r+1}}$ in $K^n_\omega$ is not covered by $(W,\sigma)$.
\end{proof}

\begin{lem}\label{lem130728z}
Assume that $r< n$, and consider an arbitrary subset $W\subseteq V$ with 
$|W|=n-r$. Then there is a function $\sigma''\colon W\to \bbn$
such that $(W,\sigma'')$ is a minimal $f$-weighted $r$-path vertex cover for $K^n_\omega$.
\end{lem}

\begin{proof}
Using the inclusion-exclusion principal, it is straightforward to show that $W$ is an $r$-path vertex cover of $K^n$.
The trivial weight $\sigma\colon W\to\bbn$ with $\sigma(v)=1$ for all $v\in W$ makes $(W,\sigma)$ into an $f$-weighted $r$-path vertex cover of $K^n_\omega$.
Lemma~\ref{prop130721} yields a minimal $f$-weighted
$r$-path  vertex cover $(W'',\sigma'')$ of $G_\omega$ such that $(W'',\sigma'')\leq(W,\sigma)$.
Lemma~\ref{lem130728a} shows that $|W''|\geq n-r=|W|$. Since $W''\subseteq W$, we must have $W=W''$, as desired.
\end{proof}

For the remainder of this section, we focus on the case $f=\max$.

\begin{prop}\label{prop130728a}
If $r\leq n$, then $\dim(S/I_{r,\max}(K^n_\omega))=r$.
\end{prop}

\begin{proof}
If $r=n$, then $I_{r,\max}(K^n_\omega)=0$ and therefore $\dim(S/I_{r,\max}(K^n_\omega))=\dim(S)=n=r$, as claimed. Assume 
for the rest of the proof that $r<n$.  Lemma~\ref{lem130728a} implies that 
for every weighted $r$-path vertex cover $(W,\sigma)$ we have $|W|\geq n-r$.  Furthermore, Lemma~\ref{lem130728z} implies that 
there is a minimal weighted $r$-path vertex cover $(W,\sigma)$ with $|W|= n-r$.
Thus, the desired conclusion follows from  Theorem~\ref{prop130725}\eqref{prop130725ii}.
\end{proof}

For the rest of the section, we focus on the case $r=2$.

The next result characterizes the weighted $3$-cliques $K^3_\omega$ such that $I_{2,\max}(K^3_{\omega})$ is Cohen-Macaulay.
Note that these cliques are key for the characterization of Cohen-Macaulayness of larger $n$-cliques in Theorem~\ref{inthm130728}.
Also, smaller $n$-cliques are very small trees that always give Cohen-Macaulay ideals;
argue as in Example~\ref{ex130831e}.

\begin{prop}\label{prop130729}
Consider a weighted 3-clique $K^3_\omega$, which we assume
by symmetry to be of the following form
$$
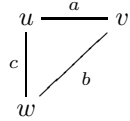
\xymatrix{
u\ar@{-}[r]^{a}\ar@{-}[d]_{c}&v\ar@{-}[dl]^{b}\\
w
}$$
with weights $a,b$, and $c$ such that $a\leq b\leq c$.
Then the following conditions are equivalent:
\begin{enumerate}[\rm(i)]
\item \label{prop130729a} The ideal $I_{2,\max}(K^3_\omega)$ is Cohen-Macaulay;
\item \label{prop130729b} The ideal $I_{2,\max}(K^3_\omega)$ is unmixed; and
\item \label{prop130729c} We have $a=b$, that is,  $a=b\leq c$.
\end{enumerate}
\end{prop}

\begin{proof}
First, we note that
\begin{align*}
I_{2,\max}(K^3_\omega)
&=(X^aY^bZ^b,X^cY^bZ^c,X^cY^aZ^c)S
=(X^aY^bZ^b,X^cY^aZ^c)S.
\end{align*}

The implication
$\eqref{prop130729a}\implies\eqref{prop130729b}$ is standard.

$\eqref{prop130729b}\implies\eqref{prop130729c}$
We argue by contrapositive. Assume that 
$a<b$.
If $a<b=c$, then it is straightforward to show that the weighted $2$-path ideal decomposes irredundantly as follows.
\begin{align*}
I_{2,\max}(K^3_\omega)
&=(X^aY^bZ^b,X^bY^aZ^b)S
=(X^a)S\bigcap(Y^a)S\bigcap(Z^b)S\bigcap(X^b,Y^b)
\end{align*}
In particular, this ideal is mixed.  When $a<b<c$, the weighted $2$-path ideal is also mixed because of the following irredundant decomposition.
\begin{align*}
I_{2,\max}(K^3_\omega)
&=(X^aY^bZ^b,X^cY^aZ^c)S\\
&=(X^a)S\bigcap(Y^a)S\bigcap(Z^b)S\bigcap(X^c,Y^b)S\bigcap(Y^b,Z^c)S
\end{align*}

$\eqref{prop130729c}\implies\eqref{prop130729a}$
If $a=b$, then we have
\begin{equation}\label{eq140801a}
I_{2,\max}(K^3_\omega)=(X^aY^aZ^a)S
\end{equation}
which is generated by a regular element and is therefore Cohen-Macaulay. 
\end{proof}

\begin{disc}\label{disc140802z}
The first display in the
proof of Proposition~\ref{prop130729} shows that the generating sequence used to define $I_{r,f}(G_\omega)$ can be redundant, i.e., non-minimal.
\end{disc}

Our next result uses the following information about colon ideals.

\begin{disc}\label{disc140802a}
Let $I$ be a monomial ideal in $S$, that is an ideal of $S$ generated by a list $g_1,\ldots,g_t$ of monomials in the variables $X_1,\ldots,X_n$.
Given another monomial $h\in S$, it is straightforward to show that
the colon ideal $(I:_Sh)$ is generated by the following list of monomials:
$g_1/\gcd(g_1,h),\ldots,g_t/\gcd(g_t,h)$.
\end{disc}

The next result contains one implication of Theorem~\ref{inthm130728} from the introduction.
Note that the 2-path Cohen-Macaulay weighted 3-cliques are characterized in Proposition~\ref{prop130729}.

\begin{thm}\label{thm130728}
Let $n\geq 3$.
Assume that  every induced weighted sub-3-clique $K^3_{\omega'}$ of $K^n_\omega$ has $I_{2,\max}(K^{3}_{\omega'})$ Cohen-Macaulay.
Then  $I_{2,\max}(K^n_\omega)$ is also Cohen-Macaulay.
\end{thm}

\begin{proof}
Set $I:=I_{2,\max}(K^n_\omega)$.
Note that our hypothesis on the induced weighted sub-3-cliques of $K^n_\omega$ imply that $I$ is generated by the following set of monomials.
$$
\{X_i^{a_{i,j,k}}X_j^{a_{i,j,k}}X_k^{a_{i,j,k}}\mid\text{$i<j<k$ and $a_{i,j,k}=\min(\omega(e_ie_j),\omega(e_ie_k),\omega(e_je_k)$}\}
$$
Indeed, this follows from Lemma~\ref{lem140801a} and the description of $I_{2,\max}(K^3_{\omega'})$ from 
equation~\eqref{eq140801a} in the proof of Proposition~\ref{prop130729}.
In particular, the generators of $I$ are determined by the induced weighted sub-3-cliques of $K^n_\omega$.

We proceed by induction on $n$.  
The base case $n=3$ is trivial.

For the inductive step, assume that $n\geq 4$ and the following: for every weighted $(n-1)$-clique $K^{n-1}_{\mu}$,
if every induced weighted sub-3-clique $K^3_{\mu'}$ of $K^{n-1}_\mu$ has $I_{2,\max}(K^{3}_{\mu'})$ Cohen-Macaulay, 
then $I_{2,\max}(K^{n-1}_\mu)$ is also Cohen-Macaulay.
Set $R:=S/I_{2,\max}(K^{n}_{\omega})$ and
$a:=\min\{\omega(v_iv_j) \mbox{ over all $i$ and $j$}\}$. Assume by symmetry that $\omega(v_1v_2)=a$.  
Let $K^{n-1}_{\omega'}$ denote the weighted sub-clique of $K^n_\omega$  induced by $V\ssm\{v_1\}$.  
Set $S'=A[X_2,\ldots,X_{n}]$.
Lemma~\ref{lem130728} implies that $R':=R/(X_1)R\cong S'/I_{2,\max}(K^{n-1}_{\omega'})$.  Since $K^{n-1}_{\omega'}$ has the same condition on the 
induced weighted sub-3-cliques, $R'$ is Cohen-Macaulay  by the inductive hypothesis.
Note that Proposition~\ref{prop130728a} says that $\dim(R')=2$.
We consider the following short exact sequence.
\begin{equation}\label{eq140801b}
0\rightarrow X_1^aR\rightarrow R\rightarrow R/X_1^aR\rightarrow 0
\end{equation}
Since $a$ is the smallest edge weight on $K^n_\omega$, we have
$R/X_1^aR\cong R'[T_1]/(T_1^a)$, which is Cohen-Macaulay of dimension 2.
As $\dim(R)=2$, in order to show that $R$ is Cohen-Macaulay, it suffices to show that $\depth(R)\geq 2$.
Applying the Depth Lemma  to the sequence~\eqref{eq140801b}, we see that
it suffices to show that $\depth_S(X_1^aR)=2$.

Case 1: Assume that $\omega(v_1v_i)=a$ for all $i=2,\dots,n$.  

Claim 1: $(I:_SX_1^a)=(X_i^aX_j^a\mid 1<i< j\leq n)S$.
For the containment $\supseteq$, let $1<i< j\leq n$.
Our assumptions on $a$ imply that the generator of $I$ corresponding to the sub-clique induced by $v_1,v_i,v_j$ is $X_1^aX_i^aX_j^a$.
It follows that the element $X_i^aX_j^a$ is in $(I:_SX_1^a)$, as desired.
For the reverse containment,
note that the generators for $I$ are of the form $X_p^\alpha X_q^\alpha X_r^\alpha$ such that $p<q<r$ and $\alpha\geq a$.  The corresponding generator of 
$(I:_SX_1^a)$ when $p=1$ is  $X_1^{\alpha-a}X_q^\alpha X_r^\alpha\in(X_q^aX_r^a)$.  When $p\neq1$ we have $X_p^\alpha X_q^\alpha X_r^\alpha\in(X_q^aX_r^a)$.  Therefore the claim holds.

Also, we have
$$X_1^aR\cong R/\ann_R(X_1^a)\cong S/(I:_SX_1^a)\cong(S'/I_{1,\max}(K^{n-1}_a))[X_1]$$
where the graph $K^{n-1}_a$ has constant weight $a$ on each edge; this is by Claim 1.  The proof of~\cite[Proposition 5.2]{paulsen:eiwg} shows that 
$S'/I_{1,\max}(K^{n-1}_a)$ is Cohen-Macaulay of dimension 1.  Therefore $X_1^aR\cong (S'/I_{1,\max}(K^{n-1}_a))[X_1]$ is Cohen-Macaulay of dimension 2.

Case 2: Assume that $\omega(v_1v_2)=a<\omega(v_1v_i)$ for some $i>2$.  
This assumption implies that there exists a subset $W\subseteq V$ such that $v_1,v_i\in W$ and for each $v_j,v_k\in W$ we have $\omega(v_jv_k)>a$.  By the finiteness of the graph $K^n$, there exists a maximal such set $W$.  Note that $|W|\geq 2$. 

Claim 2: for all $v_p\in V\ssm W$ and all $v_j\in W$, we have $\omega(v_jv_p)=a$.
Suppose by way of contradiction that $\omega(v_jv_p)>a$. 
Let $v_k\in W$ such that $v_k\neq v_j$. By assumption, we have $\omega(v_jv_k)>a$ and $\omega(v_jv_p)>a$.
Let $K^3_{\omega'}$ be the weighted sub-3-clique of $K^n_\omega$ induced by $v_j,v_k,v_p$.
By assumption, the ideal $I_{2,\max}(K^3_{\omega'})$ is Cohen-Macaulay, so Proposition~\ref{prop130729} implies that either
$\omega(v_kv_p)\geq\omega(v_jv_p)>a$ or $\omega(v_kv_p)\geq\omega(v_jv_k)>a$.
Since $v_k$ was chosen arbitrarily,  the set $W\cup\{v_p\}$ satisfies the condition for $W$, contradicting the maximality of $W$.  

Let $\lambda$ be a new weight on $K^n$ such that
$$\lambda(v_\alpha v_\beta)=\begin{cases} \omega(v_\alpha v_\beta) \mbox{ if } v_\alpha, v_\beta\in W\\ a \mbox{ if } v_\alpha\not\in W \mbox{ or } v_\beta\not\in W.\end{cases}$$
Observe that this implies for $v_j,v_k\in W$ and $v_p,v_q\not\in W$ we have
\begin{align*}
\lambda(v_jv_k)&=\omega(v_jv_k)\\
\lambda(v_jv_p)&=a=\omega(v_jv_p)\\
\lambda(v_pv_q)& \mbox{ may be different from } \omega(v_pv_q).
\end{align*}
Hence the graph $K^n_\lambda$ satisfies the induced weighted sub-3-clique assumption. 
(The four types of induced weighted sub-3-cliques are displayed next, with $v_j,v_k,v_l\in W$
and $v_p,v_q,v_r\notin W$.)
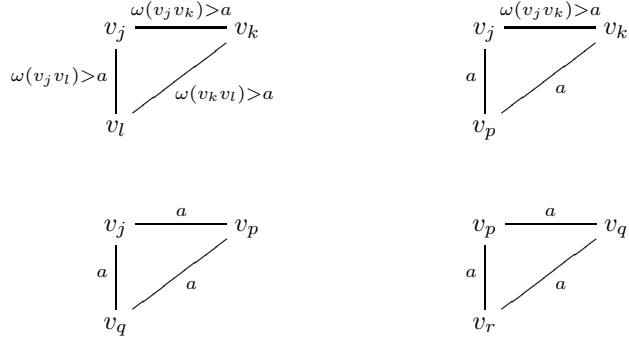
$$\xymatrix@C=12mm{
v_j\ar@{-}[r]^{\omega(v_jv_k)>a}\ar@{-}[d]_{\omega(v_jv_l)>a}&v_k\ar@{-}[dl]^{\omega(v_kv_l)>a}&&v_j\ar@{-}[r]^{\omega(v_jv_k)>a}\ar@{-}[d]_a&v_k\ar@{-}[dl]^{a}\\
v_l&&&v_p\\
v_j\ar@{-}[r]^a\ar@{-}[d]_a&v_p\ar@{-}[dl]^{a}&&v_p\ar@{-}[r]^a\ar@{-}[d]_a&v_q\ar@{-}[dl]^{a}\\
v_q&&&v_r
}$$

Since $\omega(v_1v_2)=a$, we have $v_2\not\in W$.  Thus $\lambda(v_2v_l)=a$ for all $l\neq 2$.  Hence the ideal 
$J:=I_{2,\max}(K^n_\lambda)$ is Cohen-Macaulay by Case 1.
Note that the condition $\omega(e)\geq\lambda(e)$ for each edge $e$ implies that $I\subseteq J$.

Claim 3: We have the equality $(I:_SX_1^a)=(J:_SX_1^a)$.  The containment $\subseteq$ follows from the fact that $I\subseteq J$.
For the reverse containment, recall that the generators for the ideals $I$ and $J$ are determined by the induced sub-3-cliques of $K^n$.
For the first three sub-3-cliques displayed above, the corresponding generators of $I$ and $J$ are the same.
Therefore, the generators in the colon ideals produced by these generators are the same; see Remark~\ref{disc140802a}.  
In the case of the fourth induced  sub-3-clique,  the 
associated generator for $J$ is $X_p^aX_q^aX_r^a$.
Since $p,q,r\neq 1$, the associated generator for $(J:_SX_1^a)$ is $X_p^aX_q^aX_r^a\in(X_p^aX_q^a)S\subseteq(I:_SX_1^a)$;
the last containment is explained as follows. The existence of  distinct elements $v_p,v_q,v_r\in V\ssm W$
provides a sub-3-clique induced by $v_1,v_p,v_q$, which is of the third type, with corresponding generator for the colon ideals being $X_p^aX_q^a$.
This establishes Claim~3.

Lastly, 
Case 1 shows that $\depth_S((X_1^a)S/J)=2$.
Claim 3 implies that 
$$(X_1^a)S/J\cong S/(J:_SX_1^a)=S/(I:_SX_1^a)\cong (X_1^a)S/I=X_1^aR.$$  
Therefore $\depth_S(X_1^aR)=2$, as desired.
\end{proof}

The converse  of Theorem~\ref{thm130728} is more complicated.  
We break the proof into (hopefully) manageable pieces, culminating in Theorem~\ref{thm140730}.

\begin{prop}\label{prop140729}
Let $n\geq 3$ and assume that $K^n_\omega$ contains an induced weighted sub-3-clique of the form
$$
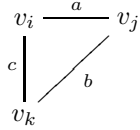
\xymatrix{
v_i\ar@{-}[r]^{a}\ar@{-}[d]_{c}&v_j\ar@{-}[dl]^{b}\\
v_k
}$$
with weights $a,b$, and $c$ such that $a< b< c$.  Then $I_{2,\max}(K^n_\omega)$ is mixed. In particular, $I_{2,\max}(K^n_\omega)$ is not Cohen-Macaulay.
\end{prop}

\begin{proof}
By symmetry, assume without loss of generality that $i=1$, $j=2$, and $k=3$.
By Theorem~\ref{prop130725}\eqref{prop130725ii}, it suffices to exhibit two minimal weighted 2-path vertex covers 
for $K^n_\omega$ whose cardinalities are not equal.
Since $\dim(S/I_{2,\max}(K^n_\omega))=2$ by Proposition~\ref{prop130728a},
we know that $K^n_\omega$ has a minimal weighted 2-path vertex cover of size $n-2$.
Thus, it suffices to find a minimal weighted 2-path vertex cover of size $n-1$.

Consider the weighted set $\{v_2^b,v_3^c,v_4^{1},\dots,v_n^{1}\}$.  
In light of the assumptions on $a$, $b$, and $c$, it is straightforward to show that this is a weighted 2-path  vertex cover
for $K^n_\omega$. We show that it gives rise to a minimal one of the form $\{v_2^b,v_3^c,v_4^{r_4},\dots,v_n^{r_n}\}$.
Since $c>b$, the weighted path 
$$\xymatrix{*+[F]{v_3^c}\ar@{-}[r]^b
&*+[F]{v_2^b}\ar@{-}[r]^a
&v_1}$$
is covered only by the weighted vertex $v_2^b$.
If the weight $b$ on this vertex were increased, then this weighted path would no longer be covered.   
Thus, the vertex $v_2$ cannot be removed from the cover, and its weight cannot be increased.
Similarly, the weighted path $v_3v_1v_2$ shows that the vertex $v_3$ cannot be removed from the cover,  and its weight cannot be increased.
Lastly, for $j\geq 4$ the weighted path $v_2v_1v_j$ is only covered by $v_j^{1}$.
Thus, the vertex $v_j$ cannot be removed from the cover; however, its weight can be increased.
\end{proof}

\begin{disc}\label{rmk140802a}
The weighted 2-path vertex cover $\{v_2^b,v_3^c,v_4^{1},\dots,v_n^{1}\}$ in the previous proof is not incredibly mysterious.
Indeed, the induced weighted sub-3-clique
$$
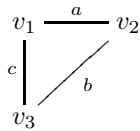
\xymatrix{
v_1\ar@{-}[r]^{a}\ar@{-}[d]_{c}&v_2\ar@{-}[dl]^{b}\\
v_3
}$$
has $\{v_2^b,v_3^c\}$ as a minimal weighted 2-path vertex cover.
(This can be checked readily as in the previous proof.
Alternately, it follows from the proof of Proposition~\ref{prop130729}; see the discussion in Example~\ref{ex130831a}.)
The given cover for $K^n_\omega$ is built from this one.

When $a<b=c$, one might guess that the vertex cover $\{v_1^b,v_2^b,v_4^{1},\dots,v_n^{1}\}$
can be used to show that $I_{2,\max}(K^n_{\omega})$ is mixed in this case as well. However, the next example shows that this is not the case.
\end{disc}

\begin{ex}\label{ex140802a}
Consider the following weighted 4-clique.
$$
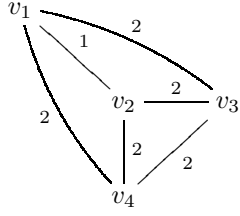
\xymatrix{
v_1\ar@{-}@/_/[ddr]_2 \ar@{-}@/^/[drr]^2 \ar@{-}[dr]^{1}\\
&v_2\ar@{-}[r]^2\ar@{-}[d]^2&v_3\ar@{-}[ld]^2\\
&v_4
}$$
It is straightforward to show that we have the following.
\begin{align*}
I_{2,\max}(K^4_\omega)
&=(X_1^{}X_2^2X_3^2, X_1^2X_2^2X_3^{},X_1^{}X_3^2X_4^2,X_1^2X_3^{}X_4^2,X_1^2X_2^2X_4^2,X_2^2X_3^2X_4^2)S\\
&=(X_1^{},X_2^2)S\bigcap(X_1^2,X_3^2)S\bigcap(X_1^{},X_4^2)S\\
&\qquad\bigcap(X_2^2,X_3^{})S\bigcap(X_2^2,X_4^2)S\bigcap(X_3^{},X_4^2)S
\end{align*}
The decomposition here shows that $I_{2,\max}(K^4_\omega)$ is unmixed. However, Theorem~\ref{thm140730} below shows that it is not Cohen-Macaulay
because the weighted sub-3-clique induced by $v_1,v_2,v_3$ is not Cohen-Macaulay; see
Proposition~\ref{prop130729}.
\end{ex}

\begin{prop}\label{prop140730}
Assume that $I_{2,\max}(K^n_\omega)$ is unmixed, and that $K^n_\omega$ has an induced
weighted sub-3-clique $K^3_{\omega'}$ such that $I_{2,\max}(K^3_{\omega'})$ is not Cohen-Macaulay.
Then $K^n_\omega$ has an induced
weighted sub-4-clique of the form
$$
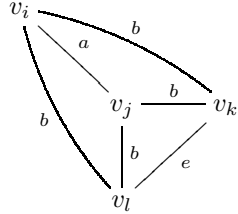
\xymatrix{
v_i\ar@{-}@/_/[ddr]_b \ar@{-}@/^/[drr]^b \ar@{-}[dr]^a\\
&v_j\ar@{-}[r]^b\ar@{-}[d]^b&v_k\ar@{-}[ld]^e\\
&v_l
}$$
such that $a<b$.
\end{prop}

\begin{proof}
Without loss of generality, assume that the non-Cohen-Macaulay induced weighted sub-3-clique is on the vertices $v_1,v_2,v_3$ 
as follows
$$
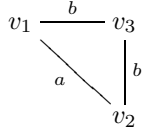
\xymatrix{
v_1\ar@{-}[dr]_{a}\ar@{-}[r]^{b}&v_3\ar@{-}[d]^{b}\\
&v_2
}$$
with $a<b$. Note that it must have this form by Propositions~\ref{prop130729} and~\ref{prop140729},
because of our unmixedness assumption.
Assume without loss of generality that $b$ is maximal among all weights occurring in a non-Cohen-Macaulay
induced sub-3-clique.

It is readily shown that the set $\{v_1^b,v_2^b,v_4^{1},\dots, v_n^{1}\}$ is a weighted 2-path vertex cover.  
As in the proof of Proposition~\ref{prop140729}, the path $v_3v_1v_2$ shows that
the vertex $v_1^b$ cannot be removed from this cover, and its weight cannot be increased.
Similarly, the path $v_1v_2v_3$ shows that
the vertex $v_2^b$ cannot be removed from this cover, and its weight cannot be increased.
Because of our unmixedness assumption, Theorem~\ref{prop130725}\eqref{prop130725ii}
and Proposition~\ref{prop130728a} imply that every minimal weighted 2-path vertex cover of $K^n_\omega$ has cardinality $n-2$.
Since the given cover has size $n-1$, one of the vertices $v_4$ through $v_n$ can be removed to create a  weighted 2-path vertex cover.  
Reorder the vertices if necessary so that $v_4$ is the vertex that can be removed.
Lemma~\ref{prop130721} shows that this gives rise to a minimal weighted 2-path vertex cover of the form
$\{v_1^b,v_2^b,v_5^{r_5},\dots, v_n^{r_n}\}$.

Label the induced weighted subgraph with  vertices $v_1,v_2,v_3,v_4$ as follows.
$$
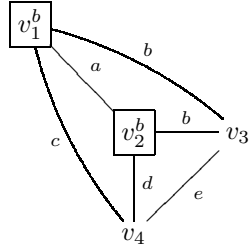
\xymatrix{
*+[F]{v_1^b}\ar@{-}@/_/[ddr]_c \ar@{-}@/^/[drr]^b \ar@{-}[dr]^a\\
&*+[F]{v_2^b}\ar@{-}[r]^b\ar@{-}[d]^d&v_3\ar@{-}[ld]^e\\
&v_4
}$$
Since $a<b$, the path $v_1v_2v_4$ must be covered by $v_2^b$.  Thus $b\leq d$.  Similarly, the vertex $v_1^b$ must cover the path $v_2v_1v_4$,
so $b\leq c$.  
Thus, we have $a<b\leq c,d$, so the  weighted sub-3-clique induced by $v_1,v_2,v_4$ is not Cohen-Macaulay.
Proposition~\ref{prop140729}  implies that $c=d$, and the maximality of $b$ implies that $c\leq b$, that is $c=b$. 
Thus, the above sub-4-clique has the desired form.
\end{proof}

The next result contains the remainder of Theorem~\ref{inthm130728} from the introduction.

\begin{thm}\label{thm140730}
Assume that $K^n_\omega$ contains at least one induced weighted sub-3-clique
$K^3_{\omega'}$ such that $I_{2,\max}(K^3_{\omega'})$ is not Cohen-Macaulay.
Then $I_{2,\max}(K^n_{\omega})$ is not Cohen-Macaulay.
\end{thm}

\begin{proof}
If $I:=I_{2,\max}(K^n_{\omega})$ is mixed, then we are done.
So, we assume that $I$ is unmixed.
Theorem~\ref{prop130725}\eqref{prop130725ii}
and Proposition~\ref{prop130728a} imply that every minimal weighted 2-path vertex cover of $K^n_\omega$ has cardinality $n-2$.
Also, Lemma~\ref{lem130728z} shows that every subset of $V$ of cardinality $n-2$ occurs as a 
minimal weighted 2-path vertex cover.

Every induced weighted sub-3-clique of $K^n_\omega$ has the form
$$
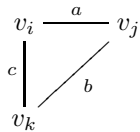
\xymatrix{
v_i\ar@{-}[r]^{a}\ar@{-}[d]_{c}&v_j\ar@{-}[dl]^{b}\\
v_k
}$$
with $a\leq b\leq c$.  
By assumption, $K^n_\omega$ contains at least one such sub-clique with $a<b\leq c$;
see Proposition~\ref{prop130729}.
Furthermore, Proposition~\ref{prop140729} implies that every such sub-clique has $a<b=c$.

Using Proposition~\ref{prop140730} and reordering the vertices if necessary, we obtain an induced weighted subgraph of the following form
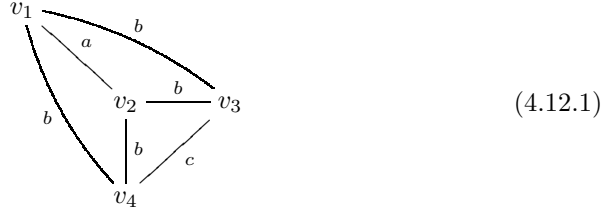
\begin{equation}
\begin{split}\label{eq140731d}
\xymatrix{
v_1\ar@{-}@/_/[ddr]_b \ar@{-}@/^/[drr]^b \ar@{-}[dr]^a\\
&v_2\ar@{-}[r]^b\ar@{-}[d]^b&v_3\ar@{-}[dl]^c\\
&v_4
}
\end{split}
\end{equation}
with $a<b$.

Using Theorem~\ref{prop130725}\eqref{prop130725ii}
we have a minimal m-irreducible decomposition
\begin{equation}\label{eq140731a}
I=\bigcap(X_{j_1}^{\beta_1},X_{j_2}^{\beta_2},\dots,X_{j_{n-2}}^{\beta_{n-2}})S
\end{equation}
where the intersection is taken over all minimal weighted 2-path vertex covers 
$\{v_{j_1}^{\beta_1},v_{j_2}^{\beta_2},\dots,v_{j_{n-2}}^{\beta_{n-2}}\}$
of $K^n_\omega$.
We set
\begin{equation}\label{eq140731b}
I_1:=\bigcap(X_1^{\alpha_1},X_{k_1}^{\alpha_{k_1}},\dots,X_{k_{n-3}}^{\alpha_{k_{n-3}}})S
\end{equation}
where the intersection is taken over all minimal weighted 2-path vertex covers of $K^n_\omega$ that contain the vertex $v_1$.
Next, set  
\begin{equation}\label{eq140731c}
I_*:=\bigcap_{j_i\not=1}(X_{j_1}^{\beta_1},X_{j_2}^{\beta_2},\dots,X_{j_{n-2}}^{\beta_{n-2}})S
\end{equation}
where the intersection is taken over all  minimal weighted 2-path vertex covers that do not contain the vertex $v_1$.
By definition, this yields $I=I_1\bigcap I_*$.
Moreover, the first paragraph of this proof implies that each of these intersections is taken over a non-empty index set.

Note that the irredundancy of the intersection in~\eqref{eq140731a} implies that the two subsequence intersections are also irredundant.
It follows that the maximal ideal $\m=(X_1,\ldots,X_n)S$ is not associated to $I_1$ and is not associated to $I_*$.
Thus, we have
$1\leq\depth(S/I_1)\leq\dim(S/I_1)=2$
and
$1\leq\depth(S/I_*)\leq\dim(S/I_*)=2$.
Since we have  $\dim(S/I)=2$, it remains to show that $\depth(S/I)=1$.

Consider the short exact sequence
$$0\rightarrow S/I\rightarrow S/I_1\oplus S/I_*\rightarrow S/(I_1+I_*)\rightarrow 0.$$
By the Depth Lemma (or a routine long-exact-sequence argument), in order to show that
$\depth(S/I)=1$, it suffices to show that $\depth(S/(I_1+I_*))=0$, that is, that $\m$ is  associated to $I_1+I_*$.

From the decompositions~\eqref{eq140731b} and~\eqref{eq140731c}, we have
\begin{equation}\label{eq140731e}
I_1+I_*=\bigcap\bigcap_{j_i\not=1}\left[(X_1^{\alpha_1},X_{k_1}^{\alpha_{k_1}},\dots,X_{k_{n-3}}^{\alpha_{k_{n-3}}})S
+(X_{j_1}^{\beta_1},X_{j_2}^{\beta_2},\dots,X_{j_{n-2}}^{\beta_{n-2}})S\right]
\end{equation}
where the first intersection is taken over all minimal weighted 2-path vertex covers that contain the vertex $v_1$,
and the second intersection is taken over all  minimal weighted 2-path vertex covers that do not contain the vertex $v_1$;
see, e.g., \cite[Lemma 2.7]{ingebretson:dmirsr}.
Note that this is an m-irreducible decomposition, though it may be redundant.
We need to show that there is an ideal in this intersection of the form
$(X_1^{\delta_1},X_2^{\delta_2},\ldots,X_n^{\delta_n})S$ that is irredundant in the intersection.

Given the sub-clique~\eqref{eq140731d},
it is straightforward to show that there are minimal weighted 2-path vertex covers of $K^n_\omega$ of the form
$\{v_1^b,v_2^b,v_5^{\alpha_5},\dots,v_n^{\alpha_n}\}$
and $\{v_3^b,v_4^b,v_5^{\beta_5},\dots,v_n^{\beta_n}\}$.
In particular, the ideal $P_1:=(X_1^b,X_2^b,X_5^{\alpha_5},\dots,X_n^{\alpha_n})S$ occurs in the 
decomposition~\eqref{eq140731b},
and the ideal $P_*:=(X_3^b,X_4^b,X_5^{\beta_5},\dots,X_n^{\beta_n})S$ occurs in the 
decomposition~\eqref{eq140731c}.
Thus, the ideal 
$$P_1+P_*=(X_1^b,X_2^b,X_3^b,X_4^b,X_5^{\gamma_5},X_6^{\gamma_6},\dots,X_n^{\gamma_n})S$$
is in the intersection~\eqref{eq140731e}, where $\gamma_i=\min\{\alpha_i,\beta_i\}$.  

Let $Q_1$ be an ideal occurring in the intersection~\eqref{eq140731b},
and let $Q_*$ be an ideal occurring in the intersection~\eqref{eq140731c}.
Suppose that 
\begin{equation}\label{eq140731f}
(X_{t_1}^{\zeta_1},X_{t_2}^{\zeta_2},\dots,X_{t_{g}}^{\zeta_{g}})S
=Q_1+Q_*\subseteq P_1+P_*
\qquad\text{with $g\leq n-1$.}
\end{equation}

Claim 1: we have
$Q_*=(X_3^{\eta_3},X_4^{\eta_4},X_5^{\eta_5},X_6^{\eta_6},\dots,X_n^{\eta_n})S$
for some $\eta_3,\ldots,\eta_n$.
By assumption, we have 
$Q_*=(X_{j_1}^{\eta_1},X_{j_2}^{\eta_2},\dots,X_{j_{n-2}}^{\eta_{n-2}})S$ 
with $j_i> 1$ for $i=1,\ldots,n-2$.
It suffices to show that $j_i\neq 2$ for all $i$.
Suppose that $j_i=2$ for some $i$. Given the conditions on the generators of $Q_*$, there must be an index $k\neq 1$ such that
$j_i\neq k$ for all $i$. 
Then $v_2^{\eta_2}$ must cover the path $v_2v_1v_k$.  This implies that $\eta_2\leq a$.  
On the other hand, since 
$$X_2^{\eta_2}\in Q_1+Q_*\subseteq P_1+P_*=(X_1^b,X_2^b,X_3^b,X_4^b,X_5^{\gamma_5},X_6^{\gamma_6},\dots,X_n^{\gamma_n})S$$ 
we have $\eta_2\geq b>a\geq \eta_2$, a contradiction.
This establishes Claim 1.

Claim 2: we have
$Q_1=(X_1^{\mu_1},X_{m_1}^{\mu_{m_1}},\dots,X_{m_{n-3}}^{\mu_{m_{n-3}}})S$ for some $\mu_1,\mu_{m_1},\dots,\mu_{m_{n-3}}$ with $m_i>2$ for all $i$.
By assumption, we have
$Q_1=(X_1^{\mu_1},X_{m_1}^{\mu_{m_1}},\dots,X_{m_{n-3}}^{\mu_{m_{n-3}}})S$ with $m_i\geq 2$.
From the equality in~\eqref{eq140731f}, we have
$$\{t_1,\ldots,t_g\}=\{3,\ldots,n\}\bigcup\{1,m_1,\ldots,m_{n-3}\}.$$
Since $g\leq n-1$, the inclusion-exclusion principle implies that
$$\left|\{3,\ldots,n\}\bigcap\{1,m_1,\ldots,m_{n-3}\}\right|\geq n-3.$$
Since $1\notin\{3,\ldots,n\}$ it follows that $m_1,\ldots,m_{n-3}\in\{3,\ldots,n\}$, that is, that $m_i>2$ for all $i$.
This establishes Claim 2.

Claim 2 says that $X_2$ does not appear to any power in the list of generators of $Q_1$.
Given the form and number of the generators of $Q_1$, it follows that there is another variable, say $X_p$ with $p\geq 3$,
that has no power occurring in this list.
By assumption, the set $\{v_1^{\mu_1},v_{m_1}^{\mu_{m_1}},\dots,v_{m_{n-3}}^{\mu_{m_{n-3}}}\}$ is a
minimal weighted  2-path vertex cover of $K^n_\omega$. 
It follows that the path $v_1v_2v_p$ is covered by $v_1^{\mu_1}$, which implies that $\mu_1\leq a$.  
However, we have $X_1^{\mu_1}\in Q_1+Q_*\subset P_1+P_*$;
as in the proof of Claim 1, this implies that $\mu_1\geq b>a\geq\mu_1$, contradiction.
We conclude that the supposition~\eqref{eq140731f} is impossible.

From this, we deduce that the only way one can have $Q_1+Q_*\subseteq P_1+P_*$
is with
$$Q_1+Q_*=(X_1^{\delta_1},X_2^{\delta_2},\ldots,X_n^{\delta_n})S$$
for some $\delta_i$. It follows that at least one ideal of this form is irredundant in the intersection~\eqref{eq140731e}, as desired.
\end{proof}

We end  with a question motivated by the results of this section.

\begin{questions}
Is there a similar characterization of the Cohen-Macaulayness of $I_{r,\max}(K^n_{\omega})$ when $r\geq 3$?
For instance, must the ideal $I_{r,\max}(K^n_\omega)$ be Cohen-Macaulay if and only if
every induced weighted sub-$(r+1)$-clique $K^{r+1}_{\omega'}$ of $K^n_\omega$ has $I_{r,\max}(K^{r+1}_{\omega'})$ Cohen-Macaulay?
\end{questions}

\section*{Acknowledgments}
We are grateful to Susan Morey for helpful conversations about this material.

\providecommand{\bysame}{\leavevmode\hbox to3em{\hrulefill}\thinspace}
\providecommand{\MR}{\relax\ifhmode\unskip\space\fi MR }
\providecommand{\MRhref}[2]{%
  \href{http://www.ams.org/mathscinet-getitem?mr=#1}{#2}
}
\providecommand{\href}[2]{#2}

\end{document}